\documentclass[12pt,reqno]{amsart}
\usepackage{graphicx}
\usepackage{amssymb,amsmath}
\usepackage{amsthm}
\usepackage{color,graphicx}
\usepackage{hyperref}
\usepackage{color}
\usepackage{epstopdf}
\usepackage{lpic}

\newcommand{\be}{\begin{equation}}
\newcommand{\ee}{\end{equation}}

\newcommand{\R}{{\mathbb R}}
\newcommand{\C}{{\mathbb C}}

\numberwithin{equation}{section}
\numberwithin{figure}{section}

\newtheorem{theorem}{Theorem}[section]
\newtheorem{proposition}[theorem]{Proposition}
\newtheorem{lemma}[theorem]{Lemma}
\newtheorem{corollary}[theorem]{Corollary}
\newtheorem{definition}[theorem]{Definition}

\theoremstyle{remark}
\newtheorem{remark}[theorem]{Remark}

\begin{document}

\title[Transverse stability of solitary waves]{On the transverse stability of smooth solitary waves in a two-dimensional Camassa--Holm equation}

\author{Anna Geyer}
\address[A. Geyer]{Delft Institute of Applied Mathematics, Faculty Electrical Engineering, Mathematics and Computer Science, Delft University of Technology, Mekelweg 4, 2628 CD Delft, The Netherlands}
\email{A.Geyer@tudelft.nl}

\author{Yue Liu}
\address[Y. Liu]{Department of Mathematics, University of Texas at Arlington, Arlington, TX 76019, USA}
\email{yliu@uta.edu}

\author{Dmitry E. Pelinovsky}
\address[D.E. Pelinovsky]{Department of Mathematics and Statistics, McMaster University,	Hamilton, Ontario, Canada, L8S 4K1}
\email{dmpeli@math.mcmaster.ca}

\begin{abstract}
	We consider the propagation of smooth solitary waves in a two-dimensional generalization of the Camassa--Holm equation. We show that  transverse perturbations to one-dimensional solitary waves behave similarly to the KP-II theory. This conclusion follows from our two main results:  
	(i) the double eigenvalue of the linearized equations related to the translational symmetry breaks under a transverse perturbation into a pair of the asymptotically stable resonances and (ii) small-amplitude solitary waves  are linearly stable with respect to transverse perturbations. 
\end{abstract}

\date{\today}
\maketitle

\section{Introduction}

The Camassa--Holm equation, labelled as {\em the CH equation}, 
\begin{equation}
\label{eq-CH}
u_t - u_{txx} +  3uu_x  = 2u_x u_{xx} + u u_{xxx},
\end{equation}
is a popular model for the dynamics of unidirectional shallow water waves \cite{CH,Johnson} which has been justified mathematically in \cite{CL}. It was originally introduced in \cite{FF} as a deformation of the integrable KdV equations. The equation models the behavior of shallow water waves both in the setting of solitary and periodic waves. Global solutions exist for initial data with sufficiently gradual slopes and wave breaking occurs in finite time for initial data with steep slopes \cite{CE-1998,CE}. There exist smooth and peaked traveling waves both among the spatially solitary and periodic waves \cite{GMNP,Lenells}. The smooth solitary waves were shown to be spectrally and orbitally stable in the time evolution of the CH equation \cite{CS2002,LP-22}. Similar stability results were obtained for the traveling periodic waves in \cite{GMNP,Len4}. On the other hand, although the peaked traveling waves (both solitary and periodic) are energetically stable in the energy space $H^1$ \cite{CM,CS,Len1,Len2}, the local solutions are only defined in the function space 
$H^1 \cap W^{1,\infty}$ \cite{DKT,Linares}. It was recently shown that 
the peaked traveling waves are both spectrally and orbitally unstable 
in $H^1 \cap W^{1,\infty}$ \cite{LP-21,MP20,NP}.

As a model for shallow water waves, the CH equation (\ref{eq-CH}) is limited to  two-dimensional fluid motion confined by a one-dimensional time-dependent surface. Transverse modulations on the water surface can be defined in terms of the two spatial variables $(x,y) \in \R^2$. A generalization of the CH equation with a two-dimensional time-dependent profile $u = u(x,y,t)$ has appeared in the literature only recently. This equation can be written in its simplest dimensionless form as 
\begin{equation}
\label{eq-CHKP}
(u_t - u_{txx} +  3uu_x - 2u_xu_{xx} - uu_{xxx})_x + u_{yy} =0.
\end{equation}
It was first derived in \cite{Chen-2006} as a model in the context of nonlinear elasticity theory. More recently, it was obtained in \cite{Liu2021} as  a model in the context of incompressible and irrotational shallow water wave theory. We refer to (\ref{eq-CHKP}) as  {\em the  CH-KP equation} because it generalizes the CH equation (\ref{eq-CH}) in the same way 
as the Kadomtsev--Petviashvili (KP) equation generalizes the classical Korteweg--de Vries (KdV) equation \cite{KP-1970}. 

In the following we review some mathematical results that have been obtained for the CH-KP equation (\ref{eq-CHKP}) so far. Local 
existence of solutions was obtained in the space of functions $X^s(\R^2)$ with $s \geq 2$, where 
$$
X^s(\R^2) := \{ u \in H^s(\R^2) : \quad \partial_x^{-1} u \in H^s(\R^2), \;\; \partial_x u \in H^s(\R^2)\},
$$
see \cite[Theorem 1.1]{Liu2021}. The nonlocal operator $\partial_x^{-1}$ can be formally defined as
$$
(\partial_x^{-1}f)(x) :=\int_{+\infty}^xf(x')\,dx'
$$
for functions $f(x) : \R\to \R$ that decay to zero as $x \to +\infty$. This nonlocal operator can be used to rewrite  (\ref{eq-CHKP}) in the evolution form
\begin{equation}
\label{eq-CHKP-evolution}
u_t + (1-\partial_x^2)^{-1} \left[ 3uu_x - 2u_xu_{xx} - uu_{xxx} + \partial_x^{-1} u_{yy} \right] = 0.
\end{equation} 
The evolution equation \eqref{eq-CHKP-evolution} can be cast 
in Hamiltonian form
\begin{equation}
	\label{eq-CHKP-Hamiltonian}
	u_t=-J F'(u),
\end{equation}
with the skew-adjoint operator $J := \partial_x (1-\partial_x^2)^{-1}$
and the conserved energy
\begin{equation}
\label{Hamiltonian}
F(u):=\frac{1}{2}\int_{\R^2} \left[ u^3+uu_x^2+(\partial_x^{-1}u_y)^2\right] \,dx\,dy.
\end{equation}
It was shown in \cite{Liu2021} that 
$F(u)$ is conserved in time for  local solutions in $X^s(\R^2)$ for $s \geq 2$, 
and so is the momentum
\begin{equation}
\label{energy}
E(u):= \frac{1}{2}\int_{\R^2}  (u^2+u_x^2)\,dx \,dy.
\end{equation}
In addition to $F(u)$ and $E(u)$, the mass 
\begin{equation}
\label{mass}
M(u) := \int_{\R^2} u \,dx \,dy
\end{equation}
is formally conserved in the time evolution of the CH-KP equation (\ref{eq-CHKP}). Various wave breaking criteria were obtained in \cite[Theorems 1.2--1.4]{Liu2021}. A recent work \cite{Liu2023} explored numerical (Galerkin) methods for approximation of solitary waves in the CH--KP equation. 

{\em The purpose of this work is to study the transverse stability of perturbed 
solitary waves in the CH-KP equation (\ref{eq-CHKP}).} Line solitary waves are obtained for functions of the form $u(x,y,t) = \phi(x + \gamma y - ct)$ with  
parameters $\gamma,c\in\R$. In what follows, we will only consider 
the case $\gamma = 0$ for the traveling wave solutions 
of the CH equation (\ref{eq-CH}). 

It was  the motivation of the pioneering work \cite{KP-1970} to investigate 
the transverse stability of solitary waves under small slowly varying 
perturbations. It was discovered that the line solitary waves are transversely unstable in one version of the KP equation and are transversely stable in another version of the KP equation. These versions are now conventionally referred to as the {\em KP-I} and {\em KP-II} equations, respectively. The CH-KP equation \eqref{eq-CHKP} we are considering in the present work corresponds to KP-II. 

A rigorous proof of transverse stability of traveling waves in the KP-II equation was  completed only recently. Linear and nonlinear stability of the solitary waves have been proven for transversely periodic perturbations in \cite{MT} and for decaying perturbations in $\mathbb{R}^2$ \cite{M}. Linear stability of traveling periodic waves was shown in \cite{Haragus} and the nonlinear stability of periodic waves is still an open problem for the KP-II equation. 

Asymptotic reductions of other nonlinear systems to the KP-II equation have been 
explored in the literature. Mizumachi and Shimabukuro used the KP--II equation as an approximation of the Benney--Luke system to prove linear and nonlinear transverse stability of the line solitary waves of small amplitudes \cite{MS1,MS2}. 
A justification of the asymptotic reduction to the KP-II equation for the two-dimensional Boussinesq equation was done by Gallay and Schneider \cite{GalSchn}. 
In the recent series of papers \cite{Gallone,Hristov,PelSchn}, the KP-II equation was justified as the leading model for a two-dimensional Fermi--Pasta--Ulam system on a square lattice. See also \cite{Panley,Bruell} for recent work on transverse stability of line solitary waves in other generalizations of the KP equation. 

We can formally obtain the asymptotic reduction of the CH--KP equation 
to the KP-II equation. 
Let $k > 0$ be a fixed parameter and consider the slowly varying approximation of small-amplitude perturbations
of a constant background in the form
\begin{equation}
\label{scaling-KP}
u(x,y,t) = k + \varepsilon^2 v(\varepsilon (x-3kt), \varepsilon^2 y, \varepsilon^3 t).
\end{equation}
By using the chain rule and the evolution form (\ref{eq-CHKP-evolution}), we derive the following evolution equation for 
the variable $v = v(X,Y,T)$ in scaled coordinates as
\begin{equation*}
v_{T} + (1- \varepsilon^2 \partial_{X}^2)^{-1} \left[ 2k v_{X X X}
+ 3 v v_{X} + \partial_{X}^{-1} v_{Y Y}  - \varepsilon^2 (2
v_{X} v_{X X} + v v_{X X X}) \right] = 0.
\end{equation*}
The formal truncation at $\varepsilon = 0$ yields the KP-II equation in the form 
\begin{equation}
\label{KP-II}
v_{T} + 2k v_{X X X}
+ 3 v v_{X} + \partial_{X}^{-1} v_{Y Y} = 0.
\end{equation}
For every fixed $k > 0$, the line solitary waves are linearly and nonlinearly stable 
in the KP-II equation (\ref{KP-II}) \cite{M}. {\em The main conclusion of this work is that the smooth solitary waves are  linearly transversely stable also in the CH-KP equation (\ref{eq-CHKP-evolution}).} The nonlinear transverse stability is still an open question, and our  results on the linear transverse stability so far are limited to two claims:
	\begin{itemize}
		\item The transverse perturbation breaks the double zero eigenvalue of the linearized equations into a pair of resonances located in the left half-plane. This result is obtained for smooth solitary waves of arbitrary amplitude.
		
		\item The line solitary waves are linearly stable with respect to transverse perturbations if the wave amplitude is  sufficiently small.
	\end{itemize}
The precise statement of these two results will be given in Section \ref{sec-2}, after the traveling waves and their linear stability problems will be described. Sections \ref{sec-3} and \ref{sec-4} contain the proofs of these two main results. Section \ref{sec-conclusion} concludes the paper with a summary and a list of open questions for further studies. 

\section{Smooth solitary waves}
\label{sec-2}

We consider the traveling one-dimensional solitary waves described by solutions to
the CH-KP equation \eqref{eq-CHKP} of the form 
$$
u(x,y,t)=\phi(x-ct),
$$ 
where $\phi(x)\to k$ as  $|x| \to \infty$, for a fixed background parameter $k > 0$. It is well-known \cite{GMNP,LP-22}, see also \cite{CS2002, Lenells} for earlier results, that such solitary waves exist for $c > 3k$ and have a smooth profile $\phi \in C^{\infty}(\R)$.  
The following lemma formalizes the result. 

\begin{lemma}
	\label{lem-solitary-wave}
	Fix $k > 0$. For every $c > 3k$, there exists a traveling solitary wave solution of the CH equation (\ref{eq-CH}) with  profile 
	$\phi \in C^{\infty}(\R)$ of the form $\phi(x) = k + \psi(x)$, where $\psi$ is found from the first-order invariant
	\begin{equation}
	\label{psi-invariant}
	(\psi')^2 = \psi^2 \;\frac{c-3k-\psi}{c-k-\psi}.
	\end{equation}
In particular, $\psi(x) > 0$ for all $x \in \R$, 
$\psi(x) \to 0$ as $|x|\to\infty$ exponentially fast, and 
$\psi(x)$ is monotonically decreasing on both sides of its maximum 
at $\max\limits_{x \in \mathbb{R}} \psi(x) = c - 3k$.  
\end{lemma}

\begin{proof}
The traveling wave of the CH equation (\ref{eq-CH}) with  profile $\phi$ satisfies the third-order differential equation
\begin{equation*}
	-c(\phi'-\phi''') + 3\phi\phi' - 2\phi'\phi'' - \phi\phi'''=0,
\end{equation*}
which can either be integrated directly to give 
\begin{equation}
	\label{eq-tws}
	(c-\phi)(\phi-\phi'')+\frac{1}{2}(\phi')^2- \frac{1}{2} \phi^2 = k c - \frac{3}{2} k^2,
\end{equation}
or first multiplied by $(c-\phi)$ and then integrated  to give 
\begin{equation}
\label{eq-two}
	-(c-\phi)^2(\phi''-\phi) = k(c-k)^2.
\end{equation}
In both cases, we have fixed the integration constant from the conditions 
$\phi(x) \to k$ and $\phi'(x), \phi''(x) \to 0$ as $|x| \to \infty$. 
Multiplying (\ref{eq-two}) by $\phi'$ and integrating again gives 
\begin{equation}
\label{first-order}
\frac{1}{2} (\phi')^2 - \frac{1}{2} \phi^2 + \frac{k (c-k)^2}{(c-\phi)} = k c - \frac{3}{2} k^2.
\end{equation}
Writing $\phi = k + \psi$, we obtain (\ref{psi-invariant}) from (\ref{first-order}).

A solitary wave with $\psi(x) \to 0$ as $|x| \to \infty$ corresponds to 
a homoclinic orbit on the phase plane $(\psi,\psi')$ along the level curve 
(\ref{psi-invariant}) to the saddle point $(0,0)$. The solitary wave exists 
if and only if $c - 3 k > 0$, because $(0,0)$ is a center point for $c - 3 k < 0$ 
and no homoclinic orbit exists for $c - 3 k = 0$. Since $(0,0)$ is a saddle point 
for $c - 3k > 0$, the convergence rate of $\psi(x) \to 0$ as $|x| \to \infty$ is exponential.  The stable and unstable curves at $(0,0)$ do not intersect if $\psi < 0$ and intersect if $\psi> 0$. Hence $\psi(x) > 0$ for all $x \in \mathbb{R}$ and the turning point $x_0 \in \mathbb{R}$ with $\psi'(x_0) = 0$ exists if and only if $\psi(x_0) = c - 3k$. Thus, the profile $\psi$ is monotonically decreasing away from its maximum at $\max\limits_{x \in \mathbb{R}} \psi(x) = c - 3k$.  
\end{proof}

\begin{remark}
	Due to the translational symmetry of the CH equation we may place the maximum of $\psi$ at $x = 0$ such that $\psi(0) = c - 3k$.
	\label{rem-maximum}
\end{remark}

\begin{remark}
	\label{remark-reductionKdV}
Since the scaling (\ref{scaling-KP}) suggests a reduction of the CH-KP equation (\ref{eq-CHKP-evolution}) to the KP-II equation  (\ref{KP-II}), the traveling solitary wave 
of Lemma \ref{lem-solitary-wave} must converge to the traveling solitary wave of the KdV equation
\begin{equation}
\label{KdV}
v_{T} + 2k v_{X X X} + 3 v v_{X} = 0.
\end{equation}
Indeed, solving the KdV equation (\ref{KdV}) for the solitary wave profile with 
$$
v(X,T) = {\rm sech}^2\left(\frac{X-T}{2 \sqrt{2k}}\right)
$$ 
gives the formal asymptotic expansion 
\begin{equation}
\label{KdV-soliton}
\phi(x) = k + \varepsilon^2 {\rm sech}^2\left(\frac{\varepsilon x}{2 \sqrt{2k}}\right) + \mathcal{O}(\varepsilon^4), \quad c = 3k + \varepsilon^2,
\end{equation}
where $\varepsilon > 0$ is an arbitrary (small) parameter and $x$ stands for $x-ct$. 
The asymptotic limit to the solitary wave of small amplitude corresponds to the limit $c \to 3k$ for which $\varepsilon \to 0$. This reduction is made rigorous in Lemma \ref{lem-approx} below.
\end{remark}

In order to set up the linear transverse stability problem for the smooth solitary 
wave of Lemma \ref{lem-solitary-wave}, we consider the decomposition 
$$
u(x,y,t) = \phi(x-ct)+v(x-ct,y,t)
$$ 
with the perturbation $v$ to the solitary wave profile $\phi \in C^{\infty}(\R)$. After substitution of the decomposition into (\ref{eq-CHKP-evolution}) and neglecting the quadratic terms in $v$, we obtain the linearized equation  
\begin{equation}
\label{eq-linHam}
v_t = J (L - \partial_x^{-2}\partial_y^2)v,
\end{equation}
where $J := \partial_x (1-\partial_x^2)^{-1}$ as in (\ref{eq-CHKP-Hamiltonian}) and 
\begin{equation}
\label{lin-oper-L}
L := c-3\phi + \phi'' - \partial_x(c-\phi)\partial_x	.
\end{equation}
Separation of variables 
in the linearized equation (\ref{eq-linHam}) by using normal modes of the form
\begin{equation*}
v(x,y,t)= e^{\lambda t} e^{i\eta y} \hat v(x),
\end{equation*}
where $\lambda \in \C$ and $\eta \in \R$, yields the spectral stability problem
\begin{equation}
\label{eq-EVproblem}
J (L +\eta^2\partial_x^{-2}) \hat v  = \lambda \hat v.
\end{equation}
The one-dimensional spectral stability problem is recovered for $\eta = 0$.
We can now specify the following definition of  transverse spectral stability. 

\begin{definition}
	We say that the solitary wave with  profile $\phi \in C^{\infty}(\R)$ 
	is transversely spectrally stable if for every $\eta \in \mathbb{R}$ there exists no eigenvalue $\lambda \in \C$ with ${\rm Re}(\lambda) > 0$ 
and eigenfunction $\hat{v} \in {\rm Dom}(J (L +\eta^2\partial_x^{-2})) \subset L^2(\R)$ of the spectral stability problem (\ref{eq-EVproblem}). 
	\label{def-stability}
\end{definition}

A common method to study the linear stability of solitary waves in the KdV equation (\ref{KdV}) is 
to use the exponentially weighted space $L_{\nu}^2$ with fixed $\nu > 0$ 
\cite{Comech,PW94}, which is defined as
\begin{equation}
\label{weighted space}
L_{\nu}^2:=\{f(x) : \R \to \R : \quad e^{\nu \cdot} f\in L^2(\R)\}.
\end{equation}
If $f \in L_{\nu}^2$ with $\nu > 0$, then $f(x) \to 0$ as $x \to +\infty$ 
and so the nonlocal operator $\partial_x^{-1}$ is well-defined. Note however that $f(x)$ does not have to decay and may even be slowly growing as $x \to -\infty$. By using the exponentially weighted space $L^2_{\nu}$, we rephrase the definition of the transverse spectral  stability. 

\begin{definition}
	We say that the solitary wave with  profile $\phi \in C^{\infty}(\R)$ 
	is transversely asymptotically stable in $L^2_{\nu}$ for some $\nu > 0$ 
	if for every $\eta \in \mathbb{R}$, $\eta \neq 0$ there exists $b > 0$ such that all points $\lambda$ in the spectrum of the linear operator 
	$$
	J (L +\eta^2\partial_x^{-2}) : {\rm Dom}(J (L +\eta^2\partial_x^{-2})) \subset L^2_{\nu} \to L^2_{\nu}
	$$ 
	satisfy ${\rm Re}(\lambda) \leq -b$. 
	\label{def-asymptoticstability}
\end{definition}

The fact that $\phi(x) \to k$ as $|x| \to \infty$ exponentially fast greatly simplifies the spectral analysis of our problem. As a result, Weyl's theory implies that the continuous 
spectrum of $J(L + \eta^2 \partial_x^{-2})$ in $L^2_{\nu}$ is uniquely determined by 
the purely continuous spectrum of $J(L_0 + \eta^2 \partial_x^{-2})$, where 
\begin{equation}
\label{lin-oper-L0}
L_0 := c - 3 k - (c-k) \partial_x^2.
\end{equation}
In addition, the point spectrum of $J(L + \eta^2 \partial_x^{-2})$ in $L^2_{\nu}$ 
may contain eigenvalues $\lambda \in \C$ with  eigenfunctions $\hat{v}\in {\rm Dom}(J(L + \eta^2 \partial_x^{-2}))$.

{\em The first result of this paper is to show that both the continuous spectrum 
	and the two eigenvalues near the origin in the complex plane satisfy 
	the transverse asymptotic stability condition of Definition \ref{def-asymptoticstability}
	for some $\nu > 0$.} The proof is developed in Section \ref{sec-3}, where the continuous spectrum is computed with the help of the Fourier transform and the two eigenvalues are computed by using Puiseux expansions \cite{W} in the small parameter $\eta$.

\begin{theorem}
	\label{theorem-splitting}
	For every $c > 3k$, $\eta \in \mathbb{R}$, and $\nu \in (0,\nu_0)$ with $\nu_0 := \sqrt{\frac{c-3k}{c-k}}$, there exists $b_0 > 0$ such that all points $\lambda$ in the spectrum of the linear operator $J(L_0 + \eta_2 \partial_x^{-2})$ in $L^2_{\nu}$ satisfy 
	${\rm Re}(\lambda) \leq -b_0$. Furthermore, there exists $\eta_0 > 0$ such that the spectrum of the linear operator 
	$J(L + \eta^2 \partial_x^{-2})$ in $L^2_{\nu}$ with $\eta \in (-\eta_0,\eta_0)$ includes a pair of simple eigenvalues $\lambda_{\pm}(\eta)$ such that for $\eta \neq 0$ we have
	\begin{itemize}
		\item ${\rm Re}(\lambda_+(\eta)) = {\rm Re}(\lambda_-(\eta)) < 0$,
		\item ${\rm Im}(\lambda_+(\eta)) = -{\rm Im}(\lambda_-(\eta)) > 0$,
	\end{itemize}
and $\lambda_+(0) = \lambda_-(0) = 0$.
\end{theorem}

\begin{remark}
The result of Theorem \ref{theorem-splitting} is consistent with the transverse asymptotic stability with respect to long transverse perturbations in the sense of Definition \ref{def-asymptoticstability} with small $\eta \neq 0$. 
However, the spectrum of $JL$ in $L^2_{\nu}$ might include more than the continuous spectrum and the double zero eigenvalue. There might exist additional embedded eigenvalues of $JL$ in $L^2(\R)$ on the imaginary axis which could become isolated in $L^2_{\nu}$ for $\nu > 0$.  The latter possibility has been 
ruled out for the KdV equation (\ref{KdV}), see \cite{PW92,PW94}. However, nothing is known about the existence of additional embedded eigenvalues of $JL$ in $L^2(\R)$ on $i\R$ for the CH equation (\ref{eq-CH}).
\end{remark}

\vspace{1cm}
 
{\em The second result of this paper explores the small-amplitude limit of 
the solitary waves and provides transverse asymptotic stability for solitary waves 
of small amplitudes in the sense of Definition \ref{def-asymptoticstability}.}
The proof is developed in Section \ref{sec-4} based on estimates for the resolvent equation.

\begin{theorem}
	\label{theorem-KP}
	Let $\lambda_{\pm}(\eta)$ be the simple eigenvalues of $J (L + \eta^2 \partial_x^{-2})$ in $L^2_{\nu}$ for fixed $\nu \in (0,\nu_0)$ found in Theorem \ref{theorem-splitting}. There exists $\varepsilon_0 > 0$ and $\beta_0 > 0$
	such that for every $\varepsilon \in (0,\varepsilon_0)$, where $\varepsilon := \sqrt{c-3k}$, and for every $\eta \in \mathbb{R}$, $\eta \neq 0$, the spectrum of $J (L + \eta^2 \partial_x^{-2})$ in $L^2_{\nu}$ is contained in 
	$$
	\mathcal{S} := \{ \lambda \in \mathbb{C}: \quad {\rm Re}(\lambda) \leq -\beta_0 \varepsilon^3 \},
	$$
with the exception of the two simple eigenvalues $\lambda = \lambda_{\pm}(\eta)$.
\end{theorem}

\begin{remark}
	Since ${\rm Re}(\lambda_{\pm}(\eta)) < 0$ for $\eta \in (-\eta_0,\eta_0)$, $\eta \neq 0$, the solitary waves of small amplitude 
	are transversely asymptotically stable in $L^2_{\nu}$. By using the Fourier transform in $y$, the result of Theorem \ref{theorem-KP} also implies the transverse asymptotic stability of these solitary waves with respect to perturbations in $L^2_{\nu}(\mathbb{R}^2)$, where the weight $\nu \in (0,\nu_0)$ is only applied in the direction of the solitary waves. This yields  linear asymptotic stability of solutions to the evolution equation (\ref{eq-linHam}) in $L^2_{\nu}(\mathbb{R}^2)$ by  semi-group theory.
\end{remark}

\section{Proof of Theorem \ref{theorem-splitting}}
\label{sec-3}

\subsection{Preliminary results}

The one-dimensional CH equation (\ref{eq-CH}) has the following conserved quantities which play a crucial role in the stability analysis of
its traveling solitary and periodic waves \cite{CS2002,GMNP}:
\begin{align*}
&\hat F(u):= \frac{1}{2}\int_{\R} (u^3 + uu_x^2 -k^3) \,dx, \\
&\hat E(u):= \frac{1}{2}\int_{\R} (u^2 + u_x^2 -k^2) \,dx, \\
&\hat M(u):=\int_{\R} (u-k) \,dx. 
\end{align*}
The constant values have been subtracted from the integrands to ensure that the integrals converge if $u(x) \to k$ as $|x| \to \infty$ sufficiently fast. These quantities are the one-dimensional analogues of the conserved quantities (\ref{Hamiltonian}), (\ref{energy}), and (\ref{mass}) 
of the two-dimensional CH-KP equation (\ref{eq-CHKP}).
Using $\hat{F}$, $\hat{E}$, and $\hat{M}$ we define the augmented energy
\begin{equation*}
\Lambda_c(u) := -\hat{F}(u) + c \hat{E}(u) - \left( ck  - \frac{3}{2} k^2 \right) \hat{M}(u).
\end{equation*} 
Smooth solutions to the second-order equation (\ref{eq-tws}) 
with the profile $\phi \in C^{\infty}(\R)$ are critical points of $\Lambda_c$ 
in the sense that the first variation vanishes:
$$
\Lambda'_c(\phi) = -\frac{3}{2} \phi^2 + \frac{1}{2} (\phi')^2 + \phi \phi'' + c \phi - c \phi'' - c k + \frac{3}{2} k^2 = 0.
$$
The linear operator $L$ in (\ref{lin-oper-L}) is the Hessian operator 
of $\Lambda_c$ at the critical point with the profile $\phi \in C^{\infty}(\R)$. This variational characterization of the traveling wave solutions was explored in the stability analysis in \cite{CS2002,GMNP}, see also \cite{LP-22} for alternative variational characterizations of the traveling wave solutions in the CH equation (\ref{eq-CH}).

\begin{remark}
Since the linear operator $L$ in (\ref{lin-oper-L}) is the Hessian operator $\Lambda_c''(\phi)$ at the traveling solitary wave with the profile $\phi \in C^{\infty}(\R)$ given by Lemma \ref{lem-solitary-wave}, it also arises in the linearization of 
the CH equation (\ref{eq-CH}) given by $v_t = JL v$. 	
\end{remark}

If $\phi = k + \psi$, then 
\begin{align}
E_{\rm 1D}(\psi) &:= \hat E(\phi) - k \hat M(\phi) \notag \\ 
&= \frac{1}{2}\int_{\R} [(k+\psi)^2+ (\psi')^2-k^2 - 2k \psi] dx \notag\\
&= \frac{1}{2} \int_{\R} [ (\psi')^2 + \psi^2] dx	
\label{eq-hatEM}
\end{align}
and 
\begin{equation}
\label{eq-hatM}
M_{\rm 1D}(\psi) := \hat M(\phi) = \int_{\R} \psi dx.
\end{equation}
The following lemma reports important monotonicity properties
of $E_{\rm 1D}(\psi)$ and $M_{\rm 1D}(\psi)$ with respect to the parameter $c \in (3k,\infty)$ for fixed $k > 0$. The proof is based on direct computations.

\begin{lemma}
	\label{lem-monotonicity}
	For fixed $k > 0$, let $\psi$ be the solitary wave defined
	by the first-order invariant (\ref{psi-invariant}). Then, the mappings
	$c \mapsto M_{\rm 1D}(\psi)$ and $c \mapsto E_{\rm 1D}(\psi)$ are monotonically
	increasing for every $c \in (3k,\infty)$.
\end{lemma}

\begin{proof}
Without  loss of generality, we place the maximum of $\psi$ at $x = 0$ such that $\psi(0) = c-3k$, see Remark \ref{rem-maximum}. By Lemma \ref{lem-solitary-wave}, we have $\psi(x) = \psi(-x) > 0$ for every $x \in \mathbb{R}$ and $\psi'(x) = - \psi'(-x) < 0$ for every $x > 0$. We obtain from (\ref{eq-hatM}) by explicit computations that 
\begin{align*}
M_{\rm 1D}(\psi) &= 2  \int_0^{\infty} \psi(x) dx \\
&= 2 \int_0^{c-3k} \frac{\sqrt{c-k-\psi}}{\sqrt{c-3k -\psi}} d \psi \\
&= 2 \int_0^{c-3k} \frac{\sqrt{2k+z}}{\sqrt{z}}dz \\
&= 8 k \int_0^{\xi_0} \sqrt{1+\xi^2}d\xi, 
\end{align*}
where we have made the substitutions $z = c - 3k - \psi$ and
$$
\xi = \frac{\sqrt{z}}{\sqrt{2k}}, \qquad \xi_0 = \frac{\sqrt{c-3k}}{\sqrt{2k}}.
$$ 
The integral is evaluated explicitly to find that
\begin{align*}
M_{\rm 1D}(\psi) = 4k \left[\xi_0\sqrt{1+\xi_0^2}+{\rm arcsinh} \xi_0 \right],
\end{align*}
from which it follows that
\begin{equation*}
\frac{d}{dc} M_{\rm 1D}(\psi) = 2\sqrt{ \frac{c-k}{c-3k}}>0.
\end{equation*}
Similarly, we find that
\begin{align*}
E_{\rm 1D}(\psi) &= 2 \int_0^{c-3k} \frac{\psi (c-2k-\psi)}{\sqrt{(c-k-\psi)(c-3k-\psi)}} d\psi \\
&= 2 \int_0^{c-3k} \frac{(c-3k-z)(z+k)}{\sqrt{z(z+2k)}}dz,
\end{align*}
from which we obtain that
\begin{align*}
\frac{d}{dc}E_{\rm 1D}(\psi) &= 2 \int_0^{c-3k}\frac{z+k}{\sqrt{z(z+2k)}}dz \\
&= 2 \sqrt{z(z+2k)} \biggr|_{z=0}^{z=c-3k} \\
&= 2\sqrt{(c-3k)(c-k)}>0.
\end{align*}
Thus, both mappings
$c \mapsto M_{\rm 1D}(\psi)$ and $c \mapsto E_{\rm 1D}(\psi)$ are monotonically
increasing for every $c \in (3k,\infty)$.
\end{proof}

\begin{remark}
The  monotonicity of $c \mapsto E_{\rm 1D}(\psi)$ plays a central role in the proof of the orbital stability of smooth solitary wave in the CH equation (\ref{eq-CH}), see \cite{CS2002}.
\end{remark}

\begin{remark}
	\label{remark-norm}
	For later reference, we also compute $\| \psi \|_{L^2}^2$ by using the same idea as in the proof of Lemma \ref{lem-monotonicity}:
	\begin{align*}
	\|\psi\|_{L^2}^2 &= 2 \int_0^{c-3k} \frac{\psi \sqrt{c - k - \psi}}{\sqrt{c-3k-\psi}} d \psi \\
	&= 2 \int_0^{c-3k} \frac{\sqrt{2k+z} (c-3k-z)}{\sqrt{z}} dz \\
	&= 8k \int_0^{\xi_0} (c-3k - 2k \xi^2) \sqrt{1 + \xi^2} d \xi, 
	\end{align*}
	from which we obtain
	\begin{align*}
\|\psi\|_{L^2}^2
	&= 4k (c-3k) \left[\xi_0\sqrt{1+\xi_0^2}+{\rm arcsinh} \xi_0\right] \\
	& \qquad
	- 2k^2 \left[ 2 \xi_0\sqrt{(1+\xi_0^2)^3}- \xi_0 \sqrt{1+\xi_0^2}-{\rm arcsinh} \xi_0 \right] \\
	&= 2k(2c-5k) \left[ \xi_0 \sqrt{1+\xi_0^2} + {\rm arcsinh} \xi_0 \right] 
	- 4k^2 \xi_0 \sqrt{(1+\xi_0^2)^3}.
	\end{align*}
\end{remark}

\subsection{The continuous spectrum of the spectral problem (\ref{eq-EVproblem})}

We start by analyzing properties of $L$. First, $L$ is a self-adjoint Sturm-Liouville  operator in $L^2(\R)$ with dense domain in $H^2(\R)$. The translational symmetry of the CH equation (\ref{eq-CH}) implies that
\begin{equation}
\label{derivative-x}
L \phi' = 0, \quad \phi' \in {\rm Dom}(L) \subset L^2(\R).
\end{equation}
Since $\phi'$ has only one zero on $\mathbb{R}$,  Sturm--Liouville theory implies that the spectrum 
of $L$ in $L^2(\mathbb{R})$ consists of one simple negative and a simple zero eigenvalue isolated from the strictly positive part of the spectrum. Furthermore, since $\phi \in C^{\infty}(\R)$ is  smooth in $c$, we find by differentiating the traveling wave equation \eqref{eq-tws} with respect to $c$ that
\begin{equation}
\label{derivative-c}
	L\partial_c\phi = k-\mu, \quad \partial_c\phi \in {\rm Dom}(L) \subset L^2(\R),
\end{equation}
where $\mu := \phi-\phi''$. Based on these computations, the following two lemmas specify properties of the linearized operator $JL$ in $L^2(\R)$, included here for the sake of completeness, and in the exponentially weighted space $L^2_{\nu}$ for small $\nu > 0$. 

\begin{lemma}
	\label{lemma-spectrum-L2}
	For every $c > 3k$, the spectrum of $JL$ in $L^2(\R)$ covers $i \R$ with $0$ being an embedded eigenvalue. 
\end{lemma}

\begin{proof}
It follows from (\ref{derivative-x}) that $JL \phi' = 0$ with $\phi' \in {\rm Ker}(JL) \subset L^2(\R)$ so that $0 \in \sigma(JL)$. Because $\phi(x) \to k$ as $|x| \to \infty$ exponentially fast, Weyl's theorem implies that the continuous spectrum of $JL$ is given by the spectrum of $JL_0$ in $L^2(\R)$, where $L_0$ is given by (\ref{lin-oper-L0}). By using the Fourier transform in $L^2(\R)$, we obtain that
	$$
	\sigma(JL_0) = \left\{ i \xi (1+\xi^2)^{-1} [ c - 3k + (c-k) \xi^2], \quad \xi \in \mathbb{R} \right\} = i \R \quad \mbox{\rm in} \;\; L^2(\R).	
	$$
Since $\phi$ is spectrally stable in the time evolution of the CH equation (\ref{eq-CH}) \cite{CS2002,LP-22}, no other points of the spectrum of $JL$ in $L^2(\R)$ exists outside $i \R$. Thus, the spectrum of $JL$ in $L^2(\R)$ is $\sigma(JL)=i \R$ with $0$ being an embedded eigenvalue.
\end{proof}

\begin{lemma}
	\label{lemma-spectrum}
	For every $c > 3k$, there exists $\nu_0 > 0$ such that the continuous spectrum of $JL$ in $L^2_{\nu}$ with $\nu \in (0,\nu_0)$ is strictly negative and the (isolated) zero eigenvalue in $L^2_{\nu}$ is algebraically double.
\end{lemma}

\begin{proof}
By Weyl's theorem, the continuous spectrum of $JL$ in $L^2_{\nu}$ is  given by the spectrum of $JL_0$ in $L^2_{\nu}$. Using the Fourier transform we obtain that
$$
	\sigma(JL_0) = \left\{ (i \xi - \nu) [1-(i\xi - \nu)^2]^{-1} [ c - 3k - (c-k) (i \xi - \nu)^2], \quad \xi \in \mathbb{R} \right\} \quad \mbox{\rm in} \;\; L^2_{\nu}.	
$$	
We claim that if  $0<\nu<\nu_0$ with $\nu_0=\sqrt{\frac{c-3k}{c-k}}$,  then 
$$
\text{Re}(\sigma(JL_0)) < 0 \quad \mbox{\rm in} \;\; L^2_{\nu},
$$ 
where $\text{Re}(\sigma(JL_0))$ coincides with 
the range of the function $\lambda_r(\xi) : \mathbb{R} \to \mathbb{R}$ given by 
\begin{align*}
\lambda_r(\xi) &= {\rm Re}\left[ (i\xi-\nu)[1-(i\xi-\nu)^2]^{-1}[c-3k-(c-k)(i\xi-\nu)^2] \right]\\
&= {\rm Re}\left[ (c-k) (i\xi-\nu) - 2k (i\xi - \nu) [1-(i\xi-\nu)^2]^{-1}  \right] \\
&= -\nu(c-k) - \frac{2k \nu (\nu^2 + \xi^2 - 1)}{(1-\nu^2+\xi^2)^2 + 4 \xi^2 \nu^2}
\end{align*}
Expanding this quantity yields 
\begin{align*}
\lambda_r(\xi) &= - \frac{\nu}{(1-\nu^2+\xi^2)^2 + 4 \xi^2 \nu^2} 
\left[ c - 3k + 2 c \xi^2 - 2(c-2k) \nu^2 + (c-k) (\xi^2 + \nu^2)^2 \right],
\end{align*}
which is strictly negative if $\nu > 0$ and 
$$
(c-k)\nu^4-2(c-2k)\nu^2+c-3k>0.
$$
The latter constraint is true if $\nu < \nu_0 =\sqrt{\frac{c-3k}{c-k}}$. Note that $\nu_0 \in (0,1)$.

It remains to prove that $0 \in \sigma(JL)$ is a double eigenvalue in $L^2_{\nu}$. Since $\phi'(x) \to 0$ as $|x| \to \infty$ exponentially fast, we have $\phi' \in L^2_{\nu}$ for sufficiently small $\nu > 0$. The Wronskian between two solutions $\{ f_1, f_2 \}$ of $Lf = 0$ is asymptotically constant at infinity and nonzero since 
$$
W(f_1,f_2) = \left| \begin{array}{cc} f_1 & f_2 \\ f_1' & f_2' \end{array} \right| = \frac{W_0}{c - \phi}, \quad x \in \mathbb{R},
$$
where $W_0$ is a nonzero constant. If one solution $f_1 := \phi'$ 
decays exponentially at infinity, the other (linearly independent) solution $f_2$ grows exponentially at infinity. Hence 
$$
\ker L = {\rm span}(\phi') \quad \mbox{\rm in} \;\; L^2_{\nu}.
$$ 
Furthermore, since $\phi$ is even, $L$ is parity preserving. There 
exists an even solution $f_0$ to the inhomogeneous equation $Lf_0  = 1$ 
and since $L$ converges to $L_0$ at infinity, $f_0$ is non-decaying at infinity. 
Since $JL f = 0$ implies $Lf = C$ for some constant $C\in\R$ and $f = C f_0 \notin L^2_{\nu}$ is non-decaying if $C \neq 0$, it follows that 
$$
\ker(JL) = \ker(L) = {\rm span}(\phi') \quad \mbox{\rm in} \;\; L^2_{\nu}.
$$
In order to study the algebraic multiplicity of the zero eigenvalue, we consider solutions of $JL f = \phi'$. Since it follows from (\ref{derivative-c}) that $JL\partial_c\phi=-\phi'$ and $\partial_c \phi \in L^2_{\nu}$, we have 
$$
\ker((JL)^2) = {\rm span}(\phi',\partial_c \phi)\quad \mbox{\rm in} \;\; L^2_{\nu}.
$$ 

The zero eigenvalue of $JL$ is algebraically double if and only if there exists no $f \in L^2_{\nu}$ such that $JL f = \partial_c \phi$, or equivalently, 
\begin{equation}
\label{lin-inhom-eq}
L f = \partial_x^{-1} \partial_c \mu, 
\end{equation}
where $\partial_x^{-1} \partial_c\mu \in L^2_{\nu}$. 
If the eigenfunctions of $L$ are defined in $L^2_{\nu}$, then the adjoint eigenfunctions are defined in $L^2_{-\nu}$ due to the transformation
$L \mapsto L_{\nu} := e^{\nu x} L e^{-\nu x}$ for eigenfunctions in the weighted space $L^2_{\nu}$, see \cite{Comech,PW94}. As a result, the inner product in $L^2_{\mu}$ is equivalent to the inner product in $L^2$, i.e.
\begin{equation}
\label{inner-product}
\forall f \in L^2_{\nu}, \;\; \forall g \in L^2_{-\nu} : \quad
\langle f,g\rangle_{L^2_{\nu}} := \langle e^{\nu x} f, e^{-\nu x} g \rangle_{L^2} = \langle f,g \rangle_{L^2}.
\end{equation}
In what follows, we drop the subscript $L^2$ for the inner product in $L^2$.
To provide the existence of solutions $f \in L^2_{\nu}$ of
the linear inhomogeneous equation (\ref{lin-inhom-eq}), 
we check the Fredholm condition
given by
\begin{equation}
\label{energy-1D}
\langle \phi', \partial_x^{-1} \partial_c \mu \rangle = - \langle (\phi - k), \partial_c \mu \rangle = - \frac{d}{dc} E_{\rm 1D}(\psi),
\end{equation}
where $E_{\rm 1D}(\psi)$ is given by (\ref{eq-hatEM}) and integration by parts gives no contribution at infinity since $\phi(x) \to k$ as $|x| \to \infty$ exponentially fast. By Lemma \ref{lem-monotonicity}, the right-hand side is strictly negative so
that no $f \in L^2_{\nu}$ exists such that $JL f = \partial_c \phi$.
Hence, $0 \in \sigma(JL)$ is a double eigenvalue in $L^2_{\nu}$.
\end{proof}

\begin{figure}[htpb!]
	\includegraphics[width=0.4\textwidth]{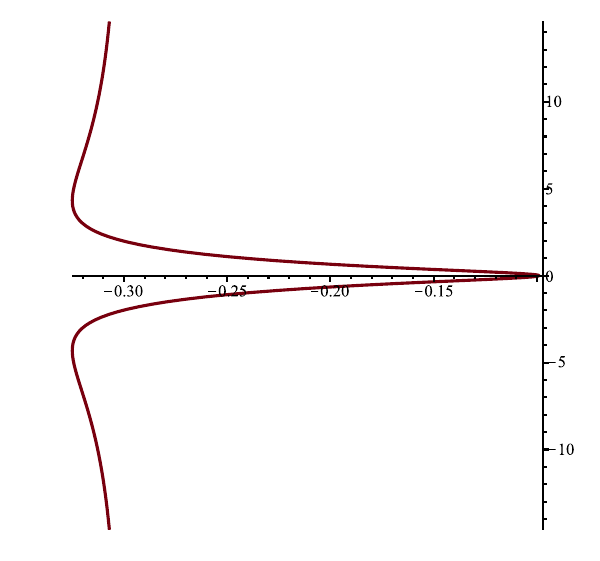}	
	\caption{A plot of $\lambda(\xi)$, $\xi \in \R$ in the complex plane for $k=1$, $c=4$, $\eta=0.01$, and $\nu=0.1$}
	\label{Fig-lambda}
\end{figure}

Based on Lemma \ref{lemma-spectrum}, we can study properties of 
the spectral stability problem (\ref{eq-EVproblem}) with transverse wave number $\eta \in \mathbb{R}$, $\eta \neq 0$. The continuous spectrum of $J(L+\eta^2 \partial_x^{-2})$ in $L^2_{\nu}$ coincides with the purely continuous spectrum of $J(L_0 + \eta^2 \partial_x^{-2})$ in $L^2_{\nu}$, which can be obtained by using the Fourier transform in $x$. The spectrum $\sigma(L_0+\eta^2 \partial_x^{-2})$ in $L^2_{\nu}$ 
is defined by the range of the function $\lambda(\xi) : \mathbb{R} \to \mathbb{C}$ given by
\begin{equation}
\label{2.10}
	\lambda(\xi) = (i\xi -\nu) [1-(i\xi - \nu)^2]^{-1} \left[ c-3k -(c-k)(i\xi - \nu)^2 + \eta^2 (i\xi - \nu)^{-2} \right].
\end{equation}
Figure \ref{Fig-lambda} gives a plot of $\lambda(\xi)$ for specific values of $k$, $c$, $\eta$, and $\nu$. The plot suggests that $\sigma(L_0 + \eta^2 \partial_x^{-2})$ in $L^2_{\nu}$ is located in the left half-plane bounded away from zero. The following lemma proves this property.

\begin{lemma}
	\label{lemma-continuous}
For every $c>3k$, $\eta \in \mathbb{R}$ and $\nu \in (0,\nu_0)$, 
where $\nu_0 := \sqrt{\frac{c-3k}{c-k}}$, we have 
$\text{\rm Re}(\lambda(\xi)) < 0$ for all $\xi\in \R$.
\end{lemma}

\begin{proof}
The expression (\ref{2.10}) can be simplified in the form: 
\begin{align*}
\lambda(\xi) =& (c-k)(i\xi-\nu)  - 2k(i\xi-\nu)  [1 - (i\xi-\nu)^2]^{-1} 
+ \eta^2 (i\xi-\nu)^{-1}  [1 - (i\xi-\nu)^2]^{-1}.
\end{align*}
Computing the real part and using $\lambda_r(\xi)$ from the proof of 
Lemma \ref{lemma-spectrum}, we obtain 
\begin{align*}
\text{\rm Re}(\lambda(\xi)) &= \lambda_r(\xi) - \frac{\eta^2 \nu (1- \nu^2 + 3 \xi^2)}{(\xi^2 + \nu^2) [(1-\nu^2+\xi^2)^2 + 4 \xi^2 \nu^2]}.
\end{align*}
Since $\lambda_r(\xi) < 0$ for $\nu \in (0,\nu_0)$ with $\nu_0 :=\sqrt{\frac{c-3k}{c-k}}$ and $\nu_0 \in (0,1)$, we have 
$\text{\rm Re}(\lambda(\xi)) < 0$ for all $\xi \in \R$.
\end{proof}

\subsection{Splitting of the double zero eigenvalue in $L^2_{\nu}$ for $\eta \neq 0$}

By Lemma \ref{lemma-spectrum}, $0$ is a double (isolated) eigenvalue 
of $JL$ in $L^2_{\nu}$ for small $\nu > 0$. When $\eta \neq 0$ in  \eqref{eq-EVproblem}, the translational symmetry is broken and the double zero eigenvalue may split into two complex eigenvalues of $J(L + \eta^2 \partial_x^{-2})$. Since it is isolated away from the continuous spectrum of $J(L + \eta^2 \partial_x^{-2})$ for every $\eta \in \R$ and small $\nu > 0$ by Lemma \ref{lemma-continuous}, the splitting can be studied by using   perturbative methods in powers of $\eta$.  

The following lemma states that when $\eta \neq 0$ the double zero eigenvalue of $JL$ in $L^2_{\nu}$ for small $\nu > 0$ splits  into a pair of eigenvalues of $J(L + \eta^2 \partial_x^{-2})$ located in the left half of the complex plane. The result holds for   solitary waves of arbitrary amplitude and is derived by means of  Puiseux expansions in $\eta$. Together with Lemma \ref{lemma-continuous}, this proves the result of Theorem \ref{theorem-splitting}.

\begin{lemma}
	\label{lemma-splitting}
	There exists $\nu_0 > 0$ such that for every fixed $\nu \in (0,\nu_0)$ there exists $\eta_0 > 0$ such that the spectrum of
	$J(L + \eta^2 \partial_x^{-2})$ in $L^2_{\nu}$ for $\eta \in (-\eta_0,\eta_0)$ 	contains a pair of simple eigenvalues $\lambda_{\pm}(\eta)$ such that 
	for $\eta \neq 0$ we have 
	\begin{itemize}
		\item ${\rm Re}(\lambda_+(\eta)) = {\rm Re}(\lambda_-(\eta)) < 0$,
		\item ${\rm Im}(\lambda_+(\eta)) = -{\rm Im}(\lambda_-(\eta)) > 0$,
	\end{itemize}
and $\lambda_+(0) = \lambda_-(0) = 0$.
\end{lemma}

\begin{proof}
	By Lemma \ref{lemma-continuous}, there exists $\nu_0 > 0$ such that
	for every fixed $\nu \in (0,\nu_0)$, the double zero eigenvalue of $JL$ in $L^2_{\nu}$ is isolated from its continuous spectrum of $J(L + \eta^2 \partial_x^{-2})$ in $L^2_{\nu}$. 
	Since $\eta^2 \partial_x^{-2}$ is a bounded analytic perturbation to
	the unbounded operator $L$ in $L^2_{\nu}$ for $\nu > 0$, the eigenvalues of 
	$J (L + \eta^2 \partial_x^{-2})$ in $L^2_{\nu}$ are continuous functions of $\eta$. 
	
	By Lemma \ref{lemma-spectrum}, the zero eigenvalue of $JL$ in $L^2_{\nu}$ is geometrically simple and algebraically double. Hence we use Puiseux expansions \cite{W} in order to trace the eigenvalues $\lambda_{\pm}(\eta)$ satisfying $\lambda_{\pm}(\eta) \to 0$ as $\eta \to 0$ with respect to small but nonzero $\eta$. Solutions of the spectral stability problem (\ref{eq-EVproblem}) with $\lambda = \lambda(\eta)$ are expanded as  
\begin{align*}
	&\hat v = v_0+  v_1\eta + v_2\eta^2  +  v_3\eta^3 + \mathcal{O}(\eta^4), \\
	&\lambda(\eta) = \lambda_1 \eta + \lambda_2 \eta^2 + \lambda_3 \eta^3 + \mathcal{O}(\eta^4). 
\end{align*}
where $v_0,v_1,v_2,v_3 \in L^2_{\nu}$ and $\lambda_1,\lambda_2,\lambda_3 \in \C$ are to be determined. We obtain at different orders in powers of $\eta$ that 
\begin{align*}
\mathcal{O}(1): &\quad  JLv_0 = 0, \\
\mathcal{O}(\eta): &\quad JLv_1  = \lambda_1 v_0, \\	
\mathcal{O}(\eta^2): &\quad JLv_2 = \lambda_2 v_0 + \lambda_1v_1 -  (1-\partial_x^2)^{-1}\partial_x^{-1}v_0  \\
\mathcal{O}(\eta^3): &\quad JLv_3 = \lambda_3 v_0 + \lambda_2 v_1 + \lambda_1 v_2 - (1-\partial_x^2)^{-1}\partial_x^{-1}v_1.
\end{align*}
With arbitrary normalization, we can set $v_0 = \phi'$ and
$v_1 = -\lambda_1 \partial_c \phi$ due to computations in the proof
of Lemma \ref{lemma-spectrum}. Then, at the order of $\mathcal{O}(\eta^2)$,
we write $v_2 = -\lambda_2 \partial_c \phi + \hat v_2$,
where $\hat v_2$ satisfies
$$
JL\hat v_2 = -\lambda_1^2 \partial_c \phi - (1-\partial_x^2)^{-1}(\phi-k).
$$
After inverting $J$ in $L^2_{\nu}$ with $\nu> 0$ we rewrite this linear inhomogeneous equation in the equivalent form
$$
L\hat v_2 = -\lambda_1^2 \partial_x^{-1}\partial_c \mu - \partial_x^{-1}(\phi-k).
$$
By using (\ref{inner-product}) we check the Fredholm condition
for the existence of solutions $\hat{v}_2 \in L^2_{\nu}$:
\begin{equation*}
	\lambda_1^2 \langle \phi',\partial_x^{-1}\partial_c\mu\rangle
	+ \langle\phi',\partial_x^{-1}(\phi-k)\rangle = 0.
\end{equation*}
Note that $\partial_x^{-1}(\phi-k) = \int_{+\infty}^x(\phi-k)$, so the second term gives after integration by parts
\begin{align*}
	\langle \phi',\partial_x^{-1}(\phi-k)\rangle
	&= (\phi-k)\int_{+\infty}^{x} (\phi-k)dx' \biggr|_{x\to-\infty}^{x \to +\infty} - \int_{-\infty}^{\infty} (\phi-k)^2dx \\
	&= -\|\phi - k\|_{L^2}^2 = - \| \psi \|^2_{L^2}.
\end{align*}
On the other hand, the first term is evaluated with the help of
(\ref{energy-1D}). Since $\frac{d}{dc} E_{\rm 1D}(\psi)>0$ by Lemma \ref{lem-monotonicity}, we obtain that
\begin{equation}
\label{lambda-1}
	\lambda_1^2 =-\frac{\langle\phi',\partial_x^{-1}(\phi-k)\rangle}{\langle \phi',\partial_x^{-1}\partial_c\mu\rangle } \\
	= - \frac{\|\psi\|_{L^2}^2}{\frac{d}{dc} E_{\rm 1D}(\psi)} < 0.
\end{equation}
Thus, we have two roots for $\lambda_1 \in i \R$, which determine
two simple eigenvalues $\lambda= \lambda_{\pm}(\eta)$. At the leading
order, we have ${\rm Im}(\lambda_+(\eta)) = - {\rm Im}(\lambda_-(\eta)) > 0$
and the complex-conjugate symmetry of eigenvalues is preserved since $J$ and $L$ are real-valued.

At the next order $\mathcal{O}(\eta^3)$ we write $v_3 = -\lambda_3 \partial_c \phi + \hat{\phi}_3$, where $\hat{v}_3$ satisfies
 \begin{equation*}
JL \hat{v}_3 = \lambda_1 \left [ \hat v_2 + (1-\partial_x^2)^{-1}\partial_x^{-1}\partial_c\phi  - 2\lambda_2 \partial_c\phi \right ],
 \end{equation*}
which, after inverting $J$ in $L^2_{\nu}$ with $\nu > 0$, gives 
\begin{equation*}
L \hat{v}_3 = \lambda_1 \left [ (1-\partial_x^2) \partial_x^{-1} \hat v_2 + \partial_x^{-2}\partial_c\phi  - 2\lambda_2\partial_x^{-1} \partial_c\mu \right ].
 \end{equation*}
By using (\ref{inner-product}) we check the Fredholm condition
for the existence of solutions $\hat{v}_3 \in L^2_{\nu}$:
\begin{align*}
	2\lambda_2 &= \frac{\langle \phi', \partial_x^{-1}\left [(1-\partial_x^2)\hat v_2 + \partial_x^{-1}\partial_c\phi\right ]\rangle}{\langle \phi',\partial_x^{-1}\partial_c\mu \rangle}\\
	&= \frac{\langle \phi -k, (1-\partial_x^2)\hat v_2\rangle +\langle \phi -k,  \partial_x^{-1}\partial_c\phi\rangle}{\frac{d}{dc}E_{\rm 1D}(\psi)}.
\end{align*}
For the first term in the numerator, we use (\ref{derivative-c}) and obtain
\begin{align*}
	\langle \phi -k, (1-\partial_x^2)\hat v_2\rangle
	&= 	\langle \mu - k, \hat v_2\rangle =  -	\langle L \partial_c \phi, \hat{v}_2 \rangle = - \langle \partial_c \phi, L \hat{v}_2 \rangle \\
	&= \lambda_1^2 \langle\partial_c \phi, \partial_x^{-1}\partial_c\mu\rangle
	+ \langle\partial_c \phi, \partial_x^{-1}(\phi -k)\rangle.
\end{align*}
We use the even parity of $\phi$ for which
$\int_{+\infty}^x (\phi - k) dx' = -\frac{1}{2} \int_{-\infty}^{\infty}(\phi - k)dx' + \int_0^x (\phi - k) dx'$, where the second term is odd, and obtain
\begin{align*}
	\langle \phi -k,  \partial_x^{-1}\partial_c\phi\rangle &=- \frac{1}{2}  \int_{-\infty}^{\infty}(\phi-k)dx
	\left(\int_{-\infty}^{\infty} \partial_c\phi dx \right)
	 = -\frac{1}{2} M_{\rm 1D}(\psi) \frac{d}{dc} M_{\rm 1D}(\psi), \\
	 \langle  \partial_c\phi, \partial_x^{-1} (\phi -k) \rangle
	 &= - \frac{1}{2}
	 \left(\int_{-\infty}^{\infty} \partial_c\phi dx \right)
	 \int_{-\infty}^{\infty}(\phi-k)dx
	 = -\frac{1}{2} M_{\rm 1D}(\psi) \frac{d}{dc} M_{\rm 1D}(\psi), \\
	 \langle\partial_c \phi, \partial_x^{-1}\partial_c\mu\rangle
	 &= - \frac{1}{2}
	 \left(\int_{-\infty}^{\infty} \partial_c\phi dx \right)
	 \int_{-\infty}^{\infty} \partial_c \mu dx
	 = -\frac{1}{2} \left( \frac{d}{dc} M_{\rm 1D}(\psi) \right)^2,
\end{align*}
which then yields
\begin{align}
	2\lambda_2
	&=\frac{\frac{d}{dc}M_{\rm 1D}(\psi)}{\frac{d}{dc}E_{\rm 1D}(\psi)} \left [ \frac{\|\psi\|_{L^2}^2}{2\frac{d}{dc}E_{\rm 1D}(\psi)} \frac{d}{dc} M_{\rm 1D}(\psi) -M_{\rm 1D}(\psi) \right ] \notag \\
	&=\frac{\frac{d}{dc}M_{\rm 1D}(\psi)}{2\left(\frac{d}{dc}E_{\rm 1D}(\psi)\right )^2} \left [ \|\psi\|_{L^2}^2 \frac{d}{dc} M_{\rm 1D}(\psi)  - 2 M_{\rm 1D}(\psi)\frac{d}{dc}E_{\rm 1D}(\psi)\right],
	\label{lambda-2}
\end{align}
where we have used (\ref{lambda-1}) for $\lambda_1^2$.

In order to identify the sign of $\lambda_2$, we recall from Lemma \ref{lem-monotonicity} that the mappings $c \mapsto M_{\rm 1D}(\psi)$
and $c \mapsto E_{\rm 1D}(\psi)$ are monotonically increasing.
Hence, the sign of $\lambda_2$ is equivalent to the sign of
\begin{align*}
& \|\psi\|_{L^2}^2 \frac{d}{dc} M_{\rm 1D}(\psi)  - 2 M_{\rm 1D}(\psi)\frac{d}{dc}E_{\rm 1D}(\psi) \\
&= \frac{4k\sqrt{c-k}}{\sqrt{c-3k}} \left[ (7k-2c) (\xi_0 \sqrt{1+\xi_0^2} + {\rm arcsinh} \xi_0) - 2k \xi_0 \sqrt{(1+\xi_0^2)^3}\right], \quad
\xi_0 := \frac{\sqrt{c-3k}}{\sqrt{2k}},
\end{align*}
where we have substituted explicit expressions from Lemma \ref{lem-monotonicity} and Remark \ref{remark-norm}.
Since $c > 3k$, we obtain
\begin{align*}
& \|\psi\|_{L^2}^2 \frac{d}{dc} M_{\rm 1D}(\psi)  - 2 M_{\rm 1D}(\psi)\frac{d}{dc}E_{\rm 1D}(\psi) \\
& \qquad \leq \frac{4k^2\sqrt{c-k}}{\sqrt{c-3k}} \left[ \xi_0 \sqrt{1+\xi_0^2} + {\rm arcsinh} \xi_0 - 2\xi_0 \sqrt{(1+\xi_0^2)^3}\right], \\
& \qquad = -\frac{4k^2\sqrt{c-k}}{\sqrt{c-3k}} \xi_0 \sqrt{1+\xi_0^2}
\left[ 1 + 2 \xi_0^2 - \frac{\log(\xi_0 + \sqrt{1+\xi_0^2})}{\xi_0 \sqrt{1+\xi_0}^2} \right],
\end{align*}
where we have used ${\rm arcsinh} \xi_0 = \log(\xi_0 + \sqrt{1+\xi_0^2})$.
Since $\log(\xi_0 + \sqrt{1+\xi_0^2}) < \xi_0 \sqrt{1 + \xi_0^2}$ for every $\xi_0 > 0$, the expression in the bracket is positive so that  $\lambda_2 < 0$. This yields  ${\rm Re}(\lambda_+(\eta)) = {\rm Re}(\lambda_-(\eta)) < 0$ at the leading order and hence for sufficiently small $\eta \neq 0$.
\end{proof}

\begin{remark}
	\label{rem-resonances}
	In the KdV limit (\ref{KdV-soliton}) as $c \to 3k$, we can simplify the expressions (\ref{lambda-1}) and (\ref{lambda-2}) 
	for $\lambda_1$ and $\lambda_2$ to obtain 
\begin{align*}
\lambda_1^2 &= -\frac{\sqrt{2k}}{2 \sqrt{c-3k}} \left[ 4 (c-3k) \xi_0 - \frac{8}{3} k \xi_0^3 + \mathcal{O}(\xi_0^5) \right] \sim -\frac{4}{3} (c-3k)
\end{align*}
	and
\begin{align*}
2\lambda_2 &= \frac{k}{(c-3k)^2} \left[ 4(3k-c) \xi_0 - \frac{8}{3} k \xi_0^3 + \mathcal{O}(\xi_0^5) \right] \sim  -\frac{8\sqrt{2k}}{3 \sqrt{c-3k}},
\end{align*}
where we have used the explicit expressions in the proof of Lemma \ref{lem-monotonicity} and the asymptotic limit $\xi_0 \to 0$. Extracting the postive square root for $\lambda_1$ yields the expansion for $\lambda_{\pm}(\eta)$ in the form
$$
\lambda_{\pm}(\eta) = \pm \frac{2i}{\sqrt{3}} \sqrt{c-3k} \eta - \frac{4}{3} \frac{\sqrt{2k}}{\sqrt{c-3k}} \eta^2 + \mathcal{O}(\eta^3).
$$
Using the KP-II scaling (\ref{scaling-KP}) and (\ref{KdV-soliton}) with 
$\eta = \varepsilon^2 \Upsilon$ and $c - 3 k = \varepsilon^2$, we obtain 
$$
\varepsilon^{-3} \lambda_{\pm}(\varepsilon^2 \Upsilon) = \pm \frac{2i}{\sqrt{3}} \Upsilon - \frac{4}{3} \sqrt{2k} \Upsilon^2 + \mathcal{O}(\Upsilon^3),
$$
which is the asymptotic expansion of the exact expression of the pair of eigenvalues $\Lambda_{\pm}(\Upsilon)$ of the corresponding linearized operator for the KP-II equation (\ref{KP-II}),
\begin{equation}
\label{resonance-KP}
\Lambda_{\pm}(\Upsilon) = \pm \frac{2i}{\sqrt{3}} \Upsilon \sqrt{1 \pm \frac{4i}{\sqrt{3}} \sqrt{2k} \Upsilon},
\end{equation}
see  \cite{M}. 
\end{remark}

\begin{remark}
	The continuous spectrum of $J(L + \eta^2 \partial_x^{-2})$ in $L^2_{\nu}$ deforms to $i \R$ as $\nu \to 0$, which can be seen by taking the limit $\nu \to 0$ in equation \eqref{2.10}. On the other hand, the location of the simple eigenvalues $\lambda_{\pm}(\eta)$ is independent of $\nu$ for 
	$\eta \in (-\eta_0,\eta_0)$ and $\nu \in (0,\nu_0)$ as follows from (\ref{lambda-1}) and (\ref{lambda-2}). As a result, 
	the continuous spectrum crosses the location of the simple eigenvalues for some $\nu_1 \in (0,\nu_0)$ that depends on $\eta \neq 0$.
	Consequently, as is shown in \cite{PW94}, the simple eigenvalues
	of $J(L + \eta^2 \partial_x^{-2})$  in $L^2_{\nu}$ for $\nu \in (\nu_1,\nu_0)$ are no longer eigenvalues of $J(L + \eta^2 \partial_x^{-2})$ in $L^2_{\nu}$ for $\nu \in (0,\nu_1)$ and in $L^2(\R)$, because they are associated with the eigenfunctions growing exponentially as $x \to -\infty$. Such points are referred to as {\em resonances} of the linear operator 
	$J(L + \eta^2 \partial_x^{-2})$, see \cite{PW94}.
\end{remark}

\section{Proof of Theorem \ref{theorem-KP}}
\label{sec-4}
	
\subsection{Preliminary results} 

We consider the spectral stability problem in the form (\ref{eq-EVproblem}). 
Writing $\phi = k + \psi$ and $c = 3k + \gamma$, we can rewrite 
the spectral problem (\ref{eq-EVproblem}) in the equivalent form
\begin{equation}
\label{spectral-prob}
\partial_x (1-\partial_x^2)^{-1} \left( 
\gamma - 3\psi + \psi'' - \partial_x(\gamma-\psi)\partial_x
 -2k \partial_x^2 + \eta^2 \partial_x^{-2} \right) \hat{v} = \lambda \hat{v}.
\end{equation}

In order to analyze the spectral problem (\ref{spectral-prob}) in the limit of  small-amplitude solitary waves, we give a rigorous proof of  the approximation result in Remark \ref{remark-reductionKdV} and justify the asymptotic approximation (\ref{KdV-soliton}). The following lemma presents this asymptotic result. 

\begin{lemma}
	\label{lem-approx}
	There exists $\varepsilon_0 > 0$ and $C_0 > 0$ such that for 
	every $\varepsilon \in (0,\varepsilon_0)$ the solitary wave solution 
	of Lemma \ref{lem-solitary-wave} satisfying $\psi(0) = c-3k$ and $\psi'(0) = 0$ can be written in the form
	\begin{equation}
	\label{KdV-approximation}
	\psi(x) = \varepsilon^2 \Psi_{\rm KdV}(X) + \varepsilon^4 \tilde{\Psi}(X), \quad X = \varepsilon x, \quad c = 3k + \varepsilon^2,
	\end{equation}
	where
	$$
	\Psi_{\rm KdV}(X) := {\rm sech}^2\left(\frac{X}{2 \sqrt{2k}}\right) \quad 
	\mbox{\rm and} \quad \| \tilde{\Psi} \|_{L^{\infty}} \leq C_0.
	$$
\end{lemma}

\begin{proof}
	Substituting $\psi(x) = \varepsilon^2 \Psi(X)$, $X = \varepsilon x$, and $c = 3k + \varepsilon^2$ into (\ref{psi-invariant}) yields the first-order invariant
	$$
	(\Psi')^2 = \Psi^2 \frac{1-\Psi}{2k + \varepsilon^2(1-\Psi)},
	$$
	for some $\Psi \in H^2(\R)$.  The function
	$\Psi_{\rm KdV}$ is a solution of the above equation in the limit $\varepsilon \to 0$. 
	To prove \eqref{KdV-approximation} we differentiate the first-order invariant and obtain the second-order equation in the form $F(\Psi,\varepsilon^2) = 0$, where 
	$F(\Psi,\varepsilon^2) : H^2(\R) \times \R \to L^2(\R)$ is the operator function given by   
	$$
	F(\Psi,\varepsilon^2) := -\Psi'' + \Psi \frac{k(2-3\Psi) + \varepsilon^2 (1-\Psi)^2}{(2k + \varepsilon^2(1-\Psi))^2}.
	$$
	It is clear that $F$ is a $C^1$ function near $(\Psi_{\rm KdV},0)$ satisfying 
	$$
	F(\Psi_{\rm KdV},0) = -\Psi_{\rm KdV}'' +  \frac{1}{4k} \Psi_{\rm KdV} (2-3\Psi_{\rm KdV}) = 0
	$$
	and 
	$$
	D_{\Psi} F(\Psi_{\rm KdV},0) = -\partial_x^2 +  \frac{1}{2k} (1-3\Psi_{\rm KdV}).
	$$
Since $0$ is a simple eigenvalue of $D_{\Psi} F(\Psi_{\rm KdV},0)$ with odd eigenfunction $\Psi_{\rm KdV}'$, and the rest of its spectrum is bounded away from $0$, the operator $D_{\Psi} F(\Psi_{\rm KdV},0)$ is invertible in the subspace of  even functions in $H^2(\R)$. By the implicit function theorem, 
there exists a unique $C^1$ mapping $\varepsilon^2 \mapsto \Psi(\cdot,\varepsilon^2) \in H^2(\R)$ which yields the unique even solution of $F(\Psi(\cdot,\varepsilon^2),\varepsilon^2) = 0$ for small $\varepsilon^2$ such that $\Psi(\cdot,\varepsilon^2) \to \Psi_{\rm KdV}$ as $\varepsilon^2 \to 0$. The decomposition (\ref{KdV-approximation}) follows from the $C^1$ property of this mapping and the continuous embedding of $H^2(\R)$ into $L^{\infty}(\R)$.
\end{proof}

The KP-II scaling (\ref{scaling-KP}) and (\ref{KdV-soliton}) corresponds to 
\begin{equation}
\label{scaling-KP-lin}
\lambda =\varepsilon^3\Lambda, \quad 
\eta = \varepsilon^2 \Upsilon, \quad 
\gamma = \varepsilon^2, \quad  
x = \varepsilon^{-1} X, \quad 
\hat{v}(x) = \hat{V}(X).
\end{equation}
By Lemma \ref{lem-approx}, we can also write 
\begin{equation}
\label{decomp-kdv}
\psi(x) = \varepsilon^2 \Psi(X), \quad \Psi := \Psi_{\rm KdV} + \varepsilon^2 \tilde{\Psi}, \quad c = 3k + \varepsilon^2.
\end{equation}
The spectral problem (\ref{spectral-prob})  can then be rewritten as
\begin{equation}
\label{equiv-1}
\partial_{X}  (1-\varepsilon^2\partial_{X}^2)^{-1} \left( L_{\rm KdV} + \varepsilon^2 L_{\rm pert} + \Upsilon^2 \partial_{X}^{-2} \right) \hat{V} = \Lambda \hat{V},
\end{equation}
where 
\begin{equation*}
L_{\rm KdV} := 1 - 3 \Psi_{\rm KdV} - 2k \partial_{X}^2, \quad 
L_{\rm pert} := \Psi'' - \partial_{X} (1-\Psi) \partial_{X} - 3 \tilde{\Psi}.
 \end{equation*}
Since 
$$
\nu_0 = \frac{\sqrt{c-3k}}{\sqrt{c-k}} = \frac{\varepsilon}{\sqrt{2k + \varepsilon^2}}
$$ 
in Lemma \ref{lemma-continuous}, we need to rescale the exponential weight $\nu$   as $\nu = \varepsilon \rho$ and replace the weighted space (\ref{weighted space}) by 
\begin{equation*}
L_{\rho}^2 := \{ F(X) : \R \to \R : \quad e^{\rho \cdot} F \in L^2(\R)\}.
\end{equation*}
The parameter $\rho$ is fixed in $(0,\rho_0)$, where 
$\rho_0 := 1/\sqrt{2k}$. In order to prove Theorem \ref{theorem-KP}, we consider the resolvent equations obtained from the spectral stability problem (\ref{spectral-prob}) in the original variables and (\ref{equiv-1}) in the scaled variables. The two resolvent 
equations are used in two different regions:
\begin{itemize}
	\item {\em the high-frequency region} with $|\eta| \geq K^2_0 \varepsilon^2$ for sufficiently large $K_0 > 0$;
	\item {\em the low-frequency region} with $|\eta| \leq K^2 \varepsilon^2$ for every fixed $K > 0$.
\end{itemize}
Combining the two regions  covers the entire range of $\eta$ values since $K$ can be taken to be greater than $K_0$. Estimates in Lemma \ref{lemma-high} and Lemma \ref{lem-low-frequency} below prove the result of Theorem \ref{theorem-KP}.

\subsection{The high-frequency region}

We start with the following result, which is a generalization of \cite[Lemma 3.1]{MS1} obtained for the linearized KP-II equation and extended here 
for the spectral problem (\ref{equiv-1}).

\begin{proposition} 
	For every $\rho \in (0,\rho_0)$ 
	there exist $\varepsilon_0 > 0$ and $\beta_0 > 0$ such that for every  $\varepsilon \in (0,\varepsilon_0)$, $\Upsilon \in \mathbb{R}$, and every $\Lambda \in \mathbb{C}$ satisfying ${\rm Re}(\Lambda) > -\beta_0$, we have 
	\begin{equation}
	\label{estimate-1}
	\| \left(\Lambda - \partial_{X} (1- \varepsilon^2 \partial_X^2)^{-1} ( 1 - (2k + \varepsilon^2) \partial_X^2 + \Upsilon^2 \partial_{X}^{-2}  ) \right)^{-1}\|_{ L_{\rho}^2 \to  L_{\rho}^2} \leq ({\rm Re}(\Lambda) + \beta_0)^{-1}. 
	\end{equation}
	Moreover, there exists $C > 0$ such that 
	\begin{align}
	& \| \partial_X (1 - \varepsilon^2 \partial_X^2)^{-1} \left(\Lambda - \partial_{X} (1- \varepsilon^2 \partial_X^2)^{-1} ( 1 - (2k + \varepsilon^2)  \partial_X^2 + \Upsilon^2 \partial_{X}^{-2}  ) \right)^{-1}\|_{ L_{\rho}^2 \to  L_{\rho}^2} \notag \\
	& \qquad 
\leq C \left({\rm Re}(\Lambda) + \beta_0 \right)^{-1/2}.\label{estimate-2}
	\end{align}
	if ${\rm Re}(\Lambda) > -\frac{1}{2} \beta_0$.
	\label{prop-estimate}	
\end{proposition}

\begin{proof}
Since the operators in the estimates (\ref{estimate-1}) and (\ref{estimate-2}) have constant coefficients, we can use the Fourier transform in $X$ and introduce the spectral function
$$
	\Lambda(\Xi) := (i \Xi - \rho) 
	[1 - \varepsilon^2 (i \Xi - \rho)^2]^{-1} 
	[1 - (2k + \varepsilon^2) (i \Xi - \rho)^2 + \Upsilon^2 (i \Xi - \rho)^{-2}],
$$
for $\Upsilon \in \R$. The function $\Lambda(\Upsilon)$ is a scaled version of the function $\lambda(\xi)$ in (\ref{2.10}). We deduce the explicit expression as in the proof 
of Lemmas \ref{lemma-spectrum} and \ref{lemma-continuous}:
\begin{align}
	{\rm Re} \left(\Lambda(\Xi)\right) &= - \rho \Big[ 1 + 
	\frac{2k(3 \Xi^2 - \rho^2 + \varepsilon^2 (\Xi^2 - \rho^2)^2)}{1 + 2 \varepsilon^2 (\Xi^2 - \rho^2) + \varepsilon^4 (\Xi^2 + \rho^2)^2} 
	\notag \\
	& \qquad 
	 + \frac{\Upsilon^2 (1 + 3 \varepsilon^2 \Xi^2 - \varepsilon^2 \rho^2)}{(\Xi^2 + \rho^2) [1 + 2 \varepsilon^2 (\Xi^2 - \rho^2) + \varepsilon^4 (\Xi^2 + \rho^2)^2] }\Big].
	 \label{Lambda-expression}
\end{align}
Since 
$$
1 - 2 \varepsilon^2 \rho^2 \leq 1 + 2 \varepsilon^2 (\Xi^2 - \rho^2) + \varepsilon^4 (\Xi^2 + \rho^2)^2 \leq [1 + \varepsilon^2 (\Xi^2 + \rho^2)]^2,
$$
we have 
\begin{align}
-{\rm Re} (\Lambda(\Xi)) & \geq \rho \left[ 1 - 2k \rho^2 
+ \frac{2k(- \varepsilon^2 \rho^4 + 3 \Xi^2 + \varepsilon^2 \Xi^4 
	+ \varepsilon^4 \rho^2 (\Xi^2 + \rho^2)^2)}{1 + 2 \varepsilon^2 (\Xi^2 - \rho^2) + \varepsilon^4 (\Xi^2 + \rho^2)^2} \right] \notag \\
& \geq \rho \left[ 1 - 2k \rho^2 - \frac{2k \varepsilon^2 \rho^4}{1 - 2 \varepsilon^2 \rho^2} + 
\frac{2k  [\Xi^2 (3 + \varepsilon^2 \Xi^2) + \varepsilon^4 \rho^2 (\Xi^2 + \rho^2)^2]}{[1 + \varepsilon^2 (\Xi^2 + \rho^2)]^2}  \right]
\label{bound-Xi}
\end{align}
uniformly for all $\Upsilon \in \R$. Therefore, there exists $\rho_0 = 1/\sqrt{2k}$ such that for every $\rho \in (0,\rho_0)$ 
there exists $\varepsilon_0 > 0$ and $\beta_0 > 0$ such that $-{\rm Re} \Lambda(\Xi) \geq \beta_0$ for every $\varepsilon \in (0,\varepsilon_0)$
uniformly for all $\Xi \in \R$. For instance, we can choose 
$$
\beta_0 := \rho \left[ 1 - 2k \rho^2 - \frac{2k \varepsilon_0^2 \rho^4}{1 - 2 \varepsilon_0^2 \rho^2} \right] > 0
$$
for a suitable choice of $\varepsilon_0 > 0$. Hence, for every $\Lambda \in \mathbb{C}$ satisfying ${\rm Re}(\Lambda) > -\beta_0$, we have 
$$
|\Lambda - \Lambda(\Xi)| \geq ({\rm Re}(\Lambda) + \beta_0)
$$
and the bound (\ref{estimate-1}) holds from standard Fourier estimates.

For the bound (\ref{estimate-2}), we obtain from (\ref{bound-Xi}) 
that there exists $\gamma_0 > 0$ such that 
 \begin{align*}
 -{\rm Re} (\Lambda(\Xi)) & \geq \beta_0 + \frac{\gamma_0   \Xi^2}{1 + \varepsilon^2 (\Xi^2 + \rho^2)}.
 \end{align*}
For instance, we can choose $\gamma_0 := 2k \rho$ since $\varepsilon_0 \rho_0 < 1$. Hence for every $\Lambda \in \mathbb{C}$ satisfying ${\rm Re}(\Lambda) > -\frac{1}{2} \beta_0$, we have 
\begin{equation}
|\Lambda - \Lambda(\Xi)| \geq \frac{1}{2} \beta_0 + \frac{\gamma_0   \Xi^2}{1 + \varepsilon^2 (\Xi^2 + \rho^2)}.
\label{one-more-bound}
\end{equation}
Since there exists $C_0 \in (0,1)$ such that 
$$
1 + 2 \varepsilon^2 (\Xi^2 - \rho^2) + \varepsilon^4 (\Xi^2 + \rho^2)^2 
\geq C_0 [1 + \varepsilon^2 (\Xi^2 + \rho^2)]^2,
$$
we obtain 
\begin{align}
\notag
\frac{|i\Xi - \rho|}{
	|1 - \varepsilon^2 (i\Xi - \rho)^2| |\Lambda - \Lambda(\Xi)|} &\leq \frac{C \sqrt{\Xi^2 + \rho^2}}{|1 + \varepsilon^2 (\Xi^2 + \rho^2)| |\Lambda - \Lambda(\Xi)|} \\
\notag
 &\leq \frac{C}{\sqrt{1 + \varepsilon^2 (\Xi^2 + \rho^2)}\sqrt{|\Lambda - \Lambda(\Xi)|}} \\
&\leq C ({\rm Re}(\Lambda) + \beta_0)^{-1/2},
\label{estimate-tech}
\end{align}
for some generic constants $C > 0$ uniformly in $\Xi \in \mathbb{R}$. The bound (\ref{estimate-2}) follows again from  Fourier theory.
\end{proof}

In order to complete the estimates in the high-frequency region, we obtain a modified 
version of Proposition \ref{prop-estimate}.

\begin{proposition} 
Let $\varepsilon_0 > 0$ and $\beta_0 > 0$ be the same as 
in Proposition \ref{prop-estimate}. There are $K_0 > 0$ and $C_0 > 0$ such 
that for every $\Lambda \in \mathbb{C}$ satisfying ${\rm Re}(\Lambda) > -\frac{1}{2} \beta_0$ and every $\Upsilon \in \R$ satisfying $|\Upsilon| \geq K_0^2$, 
we have
\begin{align}
& \| \partial_X (1 - \varepsilon^2 \partial_X^2)^{-1} \left(\Lambda - \partial_{X} (1- \varepsilon^2 \partial_X^2)^{-1} ( 1 - (2k + \varepsilon^2)  \partial_X^2 + \Upsilon^2 \partial_{X}^{-2}  ) \right)^{-1}\|_{ L_{\rho}^2 \to  L_{\rho}^2} \notag \\
& \qquad 
\leq C K_0^{-1} \left({\rm Re}(\Lambda) + \beta_0 \right)^{-1/2}.	\label{estimate-3}
\end{align}
	\label{prop-estimate-new}	
\end{proposition}

\begin{proof}
This follows from the bounds on $\Lambda(\Xi)$ obtained in the proof of Proposition \ref{prop-estimate}. If $|\Xi| \geq K_0$ and $K_0 > 0$ is sufficiently large, then it 
follows from (\ref{one-more-bound}) that for every $\Upsilon \in \mathbb{R}$, we have 
$$
|\Lambda - \Lambda(\Xi) | \geq \frac{\gamma_0 K_0^2}{1 + \varepsilon^2 (\Xi^2 + \rho^2)}.
$$
On the other hand, if $|\Xi + i \rho| \leq K_0$ and $|\Upsilon| \geq 
K_0^2 \geq K_0 |\Xi + i \rho|$, then it follows from (\ref{Lambda-expression}) 
that 
$$
|\Lambda - \Lambda(\Xi)| \geq 
\frac{\rho \Upsilon^2 (1 + 3 \varepsilon^2 \Xi^2 - \varepsilon^2 \rho^2)}{(\Xi^2 + \rho^2) [1 + \varepsilon^2 (\Xi^2 - \rho^2)]^2} \geq \frac{\rho K_0^2}{1 + \varepsilon^2 (\Xi^2 + \rho^2)}.
$$
Then, similarly to (\ref{estimate-tech}), we obtain 
\begin{equation*}
\frac{|i\Xi - \rho|}{
	|1 - \varepsilon^2 (i\Xi - \rho)^2| |\Lambda - \Lambda(\Xi)|} \leq \frac{C}{\sqrt{1 + \varepsilon^2 (\Xi^2 + \rho^2)} \sqrt{|\Lambda - \Lambda(\Xi)|}} \leq C K_0^{-1} ({\rm Re}(\Lambda) + \beta_0)^{-1/2},
\end{equation*}
for some generic constant $C > 0$ uniformly in $\Xi \in \R$. 
This justifies the bound (\ref{estimate-3-orig}).
\end{proof}

The resolvent equation in the original variables is obtained from the spectral problem (\ref{spectral-prob}) with $\gamma = \varepsilon^2$ in the form:
\begin{equation}
\label{resolvent-eq-orig}
\left (\lambda - A_0 - A_1 - A_2 \right ) u 
= f, \qquad f \in L_{\nu}^2, 
\end{equation}
where 
\begin{align*}
A_0 & :=  \partial_{x} (1-\partial_{x}^2)^{-1}( \varepsilon^2 - (2k+\varepsilon^2) \partial_x^2 + \eta^2 \partial_{x}^{-2}), \\
A_1 & := \partial_{x} (1-\partial_{x}^2)^{-1}  \partial_x \psi \partial_x, \\
A_2 & := \partial_{x} (1-\partial_{x}^2)^{-1}(-3 \psi + \psi'').
\end{align*}
Using this notation we obtain the following corollary of Proposition \ref{prop-estimate-new} which gives the bounds in original variables.
\begin{corollary}
For every $\lambda \in \mathbb{C}$ satisfying ${\rm Re}(\lambda) > -\frac{1}{2} \beta_0  \varepsilon^3$ with some $\beta_0 > 0$ and every $\eta \in \mathbb{R}$ satisfying $|\eta| \geq K_0^2 \varepsilon^2$ with sufficiently large $K_0 > 0$ we find that 
\begin{equation}
\label{estimate-1-orig}
\| (\lambda - A_0)^{-1} \|_{L^2_{\varepsilon \rho} \to L^2_{\varepsilon \rho}} \leq C \varepsilon^{-3},  
\end{equation}
\begin{equation}
\label{estimate-2-orig}
\| \partial_x (1 - \partial_x^2)^{-1} (\lambda - A_0)^{-1} \|_{L^2_{\varepsilon \rho} \to L^2_{\varepsilon \rho}} \leq C \varepsilon^{-2},  
\end{equation}
and 
\begin{equation}
\label{estimate-3-orig}
\| \partial_x (1- \partial_x^2)^{-1} (\lambda - A_0)^{-1} \|_{L^2_{\varepsilon \rho} \to L^2_{\varepsilon \rho}} \leq C K_0^{-1} \varepsilon^{-2}. 
\end{equation}
\end{corollary}  

\begin{remark}
	\label{remark-high}
Since the continuous spectrum of $\varepsilon^{-3} A_0$ in $L^2_{\rho}$ 
is bounded away from $i \R$ by the $\varepsilon$-independent constant $\beta_0$, 
and $\varepsilon^{-3} A_1$ is a relatively bounded perturbation to $\varepsilon^{-3} A_0$ of  order 
$\mathcal{O}(\varepsilon^2)$ due to the scaling (\ref{decomp-kdv}), the estimates (\ref{estimate-1-orig}), (\ref{estimate-2-orig}), and (\ref{estimate-3-orig}) apply also for 
$(\lambda - A_0 - A_1)^{-1}$ instead of $(\lambda - A_0)^{-1}$. Hence, we will use 
\begin{equation}
\label{estimate-1-new}
\| (\lambda - A_0 - A_1)^{-1} \|_{L^2_{\varepsilon \rho} \to L^2_{\varepsilon \rho}} \leq C \varepsilon^{-3},  
\end{equation}
\begin{equation}
\label{estimate-2-new}
\| \partial_x (1 - \partial_x^2)^{-1} (\lambda - A_0 - A_1)^{-1} \|_{L^2_{\varepsilon \rho} \to L^2_{\varepsilon \rho}} \leq C \varepsilon^{-2},  
\end{equation}
and 
\begin{equation}
\label{estimate-3-new}
\| \partial_x (1- \partial_x^2)^{-1} (\lambda - A_0 - A_1)^{-1} \|_{L^2_{\varepsilon \rho} \to L^2_{\varepsilon \rho}} \leq C K_0^{-1} \varepsilon^{-2}. 
\end{equation}
instead of (\ref{estimate-1-orig}), (\ref{estimate-2-orig}), and (\ref{estimate-3-orig}).
\end{remark}

The following lemma uses the fact that the operator $A_2$ in (\ref{resolvent-eq-orig}) is small compared to the operator $A_0+A_1$ in $L^2_{\varepsilon \rho}$ due to the KP-II scaling (\ref{scaling-KP-lin}) and (\ref{decomp-kdv}), see the estimate \eqref{A2small} below. As a result, we obtain the following resolvent estimate in the high-frequency region.

\begin{lemma}
	\label{lemma-high}
For every $\rho \in (0,\rho_0)$ there exists $\varepsilon_0 > 0$, $\beta_0 > 0$, and $K_0 > 0$ such that for every $\varepsilon \in (0,\varepsilon_0)$, $\eta \in \mathbb{R}$ satisfying $|\eta| \geq K_0^2 \varepsilon^2$,  and $\lambda \in \mathbb{C}$ satisfying ${\rm Re}(\lambda) > -\beta_0  \varepsilon^3$, there exists a unique solution $u \in {\rm Dom}(A_0) \subset L^2_{\varepsilon \rho}$ to the resolvent equation 
	(\ref{resolvent-eq-orig}) with $f \in L^2_{\varepsilon \rho}$ such that 
\begin{equation}
\label{bound-high}
	\| u \|_{L^2_{\varepsilon \rho}} \leq C \varepsilon^{-3} \| f \|_{L^2_{\varepsilon \rho}},
\end{equation}
	for some $C > 0$ independently of $f \in L^2_{\varepsilon \rho}$ and $\varepsilon$. 
\end{lemma}

\begin{proof}
	We use the resolvent identity 
$$
(\lambda - A_0 - A_1 - A_2)^{-1} = [I - (\lambda - A_0 - A_1)^{-1} A_2]^{-1} (\lambda - A_0 - A_1)^{-1}.
$$
It follows from the bound (\ref{estimate-1-new}) that we only need to show that the operator 
$$
I - (\lambda - A_0-A_1)^{-1} A_2
$$ 
is invertible with a bounded inverse in $L^2_{\varepsilon \rho}$,  which is true if 
$\| (\lambda - A_0 - A_1)^{-1} A_2 \|_{L^2_{\varepsilon \rho} \to L^2_{\varepsilon \rho}}$ is small. Since the decomposition (\ref{decomp-kdv}) implies that 
\begin{equation}
\label{estimate-4}
\| (-3\psi + \psi'') f \|_{L^2_{\varepsilon \rho}} \leq C \varepsilon^2 \| f \|_{L^2_{\varepsilon \rho}},
\end{equation}
it follows from the bound (\ref{estimate-2-new}) that the smallness of 
$\| (\lambda - A_0-A_1)^{-1} A_2 \|_{L^2_{\varepsilon \rho} \to L^2_{\varepsilon \rho}}$ cannot be deduced from smallness of $\varepsilon$. Nevertheless, 
if we use the estimates (\ref{estimate-3-new}) and (\ref{estimate-4}), then we obtain 
\begin{align}
\label{A2small}
\| (\lambda - A_0 - A_1)^{-1} A_2 \|_{L^2_{\varepsilon \rho} \to L^2_{\varepsilon \rho}} \leq C_0 K_0^{-1}
\end{align}
for some $C_0 > 0$. If $K_0 > 0$ is sufficiently large, the norm is small and 
the operator $I - (\lambda - A_0-A_1)^{-1} A_2$ is invertible with a bounded inverse in $L^2_{\varepsilon \rho}$. The bound (\ref{bound-high}) follows from 
(\ref{estimate-1-new}).
\end{proof}

\subsection{The low-frequency region}

We first consider the two eigenvalues $\lambda_{\pm}(\eta)$ of the spectral problem (\ref{spectral-prob}) in $L^2_{\nu}$ for small $\eta\neq 0$, see Lemma \ref{lemma-splitting}. By Remark \ref{rem-resonances}, the expansion of $\varepsilon^{-3} \lambda_{\pm}(\varepsilon^2 \Upsilon)$ in $\Upsilon$ 
agrees with the exact expression (\ref{resonance-KP}) known for the KP-II equation (\ref{KP-II}). The following lemma states that the same correspondence holds 
for every $\Upsilon$ if $\varepsilon$ is sufficiently small. 

\begin{lemma}
	\label{lem-persistence-resonance}
	Let $\Lambda_{\pm}(\Upsilon)$ be given by (\ref{resonance-KP}) for every $\Upsilon \in \R$. For every $\rho \in (0,\rho_0)$, there exists $\varepsilon_0 > 0$ 
	and $C_0 > 0$ such that for every $\varepsilon \in (0,\varepsilon_0)$ the spectral problem (\ref{equiv-1}) admits eigenvalues $\varepsilon^{-3} \lambda_{\pm}(\varepsilon^2 \Upsilon)$ in $L^2_{\rho}$ such that 
	$$
	| \varepsilon^{-3} \lambda_{\pm}(\varepsilon^2 \Upsilon) - \Lambda_{\pm}(\Upsilon)| \leq C_0 \varepsilon^2.
	$$
\end{lemma}

\begin{proof}
By bootstrapping arguments, an eigenfunction $\hat{V}$ of the spectral problem (\ref{equiv-1}) in $L^2_{\rho}$ satisfies that 
$$
\hat{V} \in {\rm Dom}(\partial_X (1-\varepsilon^2 \partial_X^2)^{-1} (L_{\rm KdV} + \Upsilon^2 \partial_X^{-2})) \subset L^2_{\rho}
$$
if and only if 
$$
\hat{V} \in {\rm Dom}(\partial_X (L_{\rm KdV} + \Upsilon^2 \partial_X^{-2})) \subset L^2_{\rho}.
$$ 
Hence we can rewrite the spectral problem (\ref{equiv-1}) for the eigenfunction $\hat{V}$ in $L^2_{\rho}$ in the equivalent form 
\begin{equation}
\label{equiv-2}
\partial_{X}  \left( L_{\rm KdV} + \varepsilon^2 L_{\rm pert} + \Upsilon^2 \partial_{X}^{-2} \right) \hat{V} = \Lambda (1-\varepsilon^2\partial_{X}^2) \hat{V}.
\end{equation}
Since $(\Lambda_{\pm}(\Upsilon),U_{\pm}) \in \mathbb{C} \times L^2_{\rho}$ 
are solutions of the truncated problem 
\begin{equation}
\label{equiv-3}
\partial_{X}  \left( L_{\rm KdV} + \Upsilon^2 \partial_{X}^{-2} \right) U_{\pm} = \Lambda_{\pm}(\Upsilon) U_{\pm},
\end{equation}
we can write the decomposition $\Lambda = \Lambda_{\pm}(\Upsilon) + \varepsilon^2 \tilde{\Lambda}$, $\hat{V} = U_{\pm} + \varepsilon \tilde{U}$ and obtain the perturbed problem for $(\tilde{\Lambda},\tilde{U})$ given by
\begin{align*}
& \partial_{X}  \left( L_{\rm KdV} + \varepsilon^2 L_{\rm pert} + \Upsilon^2 \partial_{X}^{-2} \right) \tilde{U} - (\Lambda_{\pm}(\Upsilon) + \varepsilon^2 \tilde{\Lambda}) (1-\varepsilon^2\partial_{X}^2) \tilde{U} \\
& \quad = -L_{\rm pert} U_{\pm} - \Lambda_{\pm} \partial_X^2 U_{\pm} + \tilde{\Lambda} (1-\varepsilon^2 \partial_X^2) U_{\pm}.
\end{align*}
This equation is routinely solved by using the method of Lyapunov--Schmidt reduction 
with $\tilde{\Lambda}$ being uniquely defined from the condition that 
$\tilde{U} \in {\rm Dom}(\partial_X (L_{\rm KdV} + \Upsilon^2 \partial_X^{-2})) \subset L^2_{\rho}$ satisfy the orthogonality condition 
to the adjoint eigenfunction for the eigenvalue $\Lambda_{\pm}(\Upsilon)$. 
See Lemma 3.4 and Corollary 3.5 in \cite{MS1} for details.
\end{proof}
 
The resolvent equation in the scaled variables is obtained   from the spectral stability problem (\ref{equiv-1}) in the form
 \begin{equation}
 \label{resolvent-eq}
 \left (\Lambda - \partial_{X} (1-\varepsilon^2\partial_{X}^2)^{-1} \left( L_{\rm KdV} + \varepsilon^2 L_{\rm pert} + \Upsilon^2 \partial_{X}^{-2}  \right)\right ) U 
 = F, \qquad F \in L_{\rho}^2.
 \end{equation}
 The following lemma uses the smallness of $\varepsilon^2 L_{\rm pert}$ 
 and the formalism from \cite{MS1} in order to obtain the resolvent estimate 
 in the low-frequency region.
 
\begin{lemma}
	\label{lem-low-frequency}
For every $\rho \in (0,\rho_0)$ there exists $\varepsilon_0 > 0$, $\beta_0 > 0$ 
such that for every $\varepsilon \in (0,\varepsilon_0)$, $\Upsilon \in \mathbb{R}$ and $\Lambda \in \mathbb{C}$ satisfying ${\rm Re}(\Lambda) > -\beta_0$ and $\Lambda \neq \varepsilon^{-3} \lambda_{\pm}(\varepsilon^2 \Upsilon)$, there exists a unique solution 
 $$
 U \in {\rm Dom}(\partial_X (1-\varepsilon^2 \partial_X^2)^{-1} (L_{\rm KdV} + \Upsilon^2 \partial_X^{-2})) \subset L^2_{\rho}
 $$ 
of the resolvent equation (\ref{resolvent-eq}) for every $F \in L^2_{\rho}$ satisfying
\begin{equation}
\label{resolvent-estimate-last}
 \| U \|_{L^2_{\rho}} \leq C \| F \|_{L^2_{\rho}}
\end{equation} 
for $C > 0$.
\end{lemma} 

\begin{proof}
Let $Q_{\rm KP}$ be the projection operator for the spectral problem (\ref{equiv-3}) 
which reduces $L^2_{\rho}$ to the subspace orthogonal to the two adjoint eigenfunctions 
for the eigenvalues $\Lambda_{\pm}(\Upsilon)$. It follows from 
Proposition 3.2 in \cite{MS1} (proven in \cite{M}) that 
there exists $\beta_0 > 0$ and $C_0 > 0$ such that for every 
$\Lambda \in \C$ satisfying ${\rm Re}(\Lambda) > -\beta_0$  
and every $F \in L^2_{\rho}$, we have 
\begin{equation}
\label{resolvent-estimate-1}
\| (\Lambda - \partial_{X} (L_{\rm KdV} + \Upsilon^2 \partial_{X}^{-2}))^{-1} 
Q_{\rm KP} F \|_{L_{\rho}^2} \leq C_0 \| F \|_{L^2_{\rho}}.
\end{equation}
By the proximity result of Lemma \ref{lem-persistence-resonance}, we can introduce $\mathcal{Q}$, the projection operator for the spectral problem (\ref{equiv-2}) which reduces $L^2_{\rho}$ to 
the subspace orthogonal to the two adjoint 
eigenfunctions for the eigenvalues $\varepsilon^{-3} \lambda_{\pm}(\varepsilon^2 \Upsilon)$. The bound (\ref{resolvent-estimate-1}) and the proximity result suggest that there exists $\beta_0 > 0$ and $C_0 > 0$ such that for every $\Lambda \in \C$ satisfying ${\rm Re}(\Lambda) > -\beta_0$  
and every $F \in L^2_{\rho}$, we have 
\begin{equation}
\label{resolvent-estimate-2}
\| (\Lambda - \mathcal{M})^{-1} 
\mathcal{Q} F \|_{L_{\rho}^2} \leq C_0 \| F \|_{L^2_{\rho}},
\end{equation}
where 
$$
\mathcal{M} :=\partial_{X} (1-\varepsilon^2\partial_{X}^2)^{-1} (L_{\rm KdV} - \varepsilon^2 \partial_X (1 - \Psi) \partial_X + \Upsilon^2 \partial_{X}^{-2}).
$$
Writing again the resolvent identity as 
\begin{align*}
& (\Lambda - \partial_{X} (1-\varepsilon^2\partial_{X}^2)^{-1} (L_{\rm KdV} + \varepsilon^2 L_{\rm pert}  + \Upsilon^2 \partial_{X}^{-2}))^{-1})^{-1} \\
& \qquad = [I - \varepsilon^2 (\Lambda - \mathcal{M})  (\Psi'' - 3 \tilde{\Psi})]^{-1}
(\Lambda - \mathcal{M})^{-1}
\end{align*}
and using smallness of $\varepsilon^2$, we obtain the invertibility of the 
near-identity operator 
$$
[I - \varepsilon^2 (\Lambda - \mathcal{M})  (\Psi'' - 3 \tilde{\Psi})] : L^2_{\rho} \to L^2_{\rho}
$$
for every $\Lambda \in \C$ satisfying ${\rm Re}(\Lambda) > -\beta_0$.  
The bound (\ref{resolvent-estimate-last}) on the unique solution $U$ to the resolvent equation (\ref{resolvent-eq}) follows from the bound (\ref{resolvent-estimate-2}).
\end{proof}

\section{Conclusion}
\label{sec-conclusion}

We have derived two results, which suggest that the transverse perturbations 
to the one-dimensional solitary waves of the CH equation (\ref{eq-CH}) are stable in the time evolution of the CH-KP equation (\ref{eq-CHKP}),  similar to the KP-II theory. First, we proved that the double zero  eigenvalue of the linearized equation related to the translational symmetry breaks under a transverse perturbation into a pair of the asymptotically stable resonances, which are isolated eigenvalues in the exponentially weighted $L^2$ space. Second, we considered the small-amplitude solitary waves governed by the perturbed KP-II equation and proved their linear stability under transverse perturbations. \\

We conclude the paper with a list of further questions. 
First, nonlinear stability of  small-amplitude solitary waves of CH-KP is an open question, see \cite{MS2} for such analysis in the Benney--Luke equation. Second, peaked traveling waves of the CH equation (\ref{eq-CH}) exist but they 
are linearly and nonlinearly unstable in the time evolution in $H^1(\R) \cap W^{1,\infty}(\R)$, see \cite{LP-21,NP}. It would be  interesting to see how 
the peaked profile of the solitary waves breaks under transverse perturbations and whether cusps (waves with infinite slopes at their maximum) would form  in finite time. 
Third, transverse stability of smooth periodic waves and transverse instability of peaked periodic waves 
can be studied based on the stability analysis of the periodic waves in the one-dimensional model, see \cite{GMNP} and \cite{MP20}. 
Finally, hydrodynamical applications of the obtained results are interesting in their own right within modeling of shallow water waves in seas and oceans \cite{Liu2021}.  

\vspace{0.5cm}

{\bf Acknowledgement.} This project was started in June 2022 during a Research in Teams stay at the Erwin Schr\"{o}dinger Institute, Vienna. The authors thank members of stuff of the ESI for support during this work. D. E. Pelinovsky acknowledges the funding of this study provided by Grants No. FSWE-2020-0007 and No. NSH-70.2022.1.5.

\bibliographystyle{siam}

\begin{thebibliography}{99}
	
	\bibitem{Panley} Bhavna, A. K. Pandey, S. Singh, ``Transverse spectral 
	instabilities in Konopelchenko--Dubrovsky equation", Stud. Appl. Math. (2023) in print.
	
	\bibitem{Bruell} H. Borluk, G. Bruell, and D. Nilsson, ``Traveling waves and transverse instability forthe fractional Kadomtsev–Petviashvili equation", 
	Stud. Appl. Math. {\bf 149} (2022) 95--123.
		
	\bibitem{CH} R. Camassa and D.D. Holm, ``An integrable shallow water equation with peaked solitons", \textit{Phys. Rev. Lett.} {\bf 71} (1993), 1661--1664.
	
	\bibitem{Chen-2006} R.M. Chen, ``Some nonlinear dispersive waves arising in compressible hyperelastic plates", \textit{Int. J. Eng. Sci.} {\bf 44} (2006) 1188--1204.
	
	\bibitem{Comech} A. Comech, S. Cuccagna, and D. Pelinovsky, ``Nonlinear instability of a critical traveling wave in the generalized Korteweg-de Vries equation", \textit{SIAM J. Math. Anal.} {\bf 39} (2007) 1--33.
	
	\bibitem{CE-1998} A. Constantin and J. Escher, ``Wave breaking for nonlinear nonlocal shallow water equations", \textit{Acta Math.} 
	{\bf 181} (1998), 229--243.
	
	\bibitem{CE} A. Constantin and J. Escher, ``Well-posedness, global existence, and blowup phenomena for a periodic quasi-linear hyperbolic equation", 
	\textit{Comm. Pure Appl. Math.} {\bf 51} (1998), 475--504. 
	
	\bibitem{CL} A. Constantin and D. Lannes, ``The hydrodynamical relevance of the Camassa--Holm and Degasperis--Procesi equations", 
	\textit{Arch. Ration. Mech. Anal.} {\bf 192} (2009), 165--186.
	
	\bibitem{CM} A. Constantin and L. Molinet, ``Orbital stability of solitary waves for a shallow water equation", Physica D {\bf 157} (2001) 75--89.
	
	\bibitem{CS} A. Constantin and W.A. Strauss, ``Stability of peakons", 
	\textit{Comm. Pure Appl. Math.} {\bf 53} (2000), 603--610.
	
	\bibitem{CS2002}  A. Constantin and W.A. Strauss, ``Stability of the Camassa--Holm solitons”, \textit{J. Nonlinear Sci.} {\bf 12} (2002), 415--422.
	
	\bibitem{DKT} C. De Lellis, T. Kappeler, and P. Topalov, ``Low-regularity solutions of the periodic
	Camassa--Holm equation", \textit{Comm. PDEs} {\bf 32} (2007), 87--126.
	
	\bibitem{FF} A. Fokas and B. Fuchssteiner,  ``Symmpletic structures, their Backlund transform and hereditary symmetries",   Physica D 
	{\bf 4} (1981), 47--66
	
	\bibitem{GalSchn} T. Gallay and G. Schneider, ``KP description of unidirectional long waves. The model case," \textit{Proc. R. Soc. Edinburgh A} {\bf 131} (2001), 885--898.
	
	\bibitem{Gallone} M. Gallone and S. Pasquali, ``Metastability phenomena in two-dimensional rectangular lattices with nearest-neighbour interaction", \textit{Nonlinearity} {\bf 34} (2021), 4983--5044.
	
	\bibitem{GMNP}  A. Geyer, R.H. Martins, F. Natali, and D.E. Pelinovsky, ``Stability of smooth periodic traveling waves in the Camassa-Holm equation",
	\textit{Stud. Appl. Math.} {\bf 148} (2022) 27--61.
	
	\bibitem{Liu2021} G. Gui, Y. Liu, W. Luo, and Z. Yin, ``On a two-dimensional nonlocal shallow-water model", \textit{Adv. Math.} {\bf 392} (2021) 108021 (44 pages).	
	
	\bibitem{Haragus} M. Haragus, J. Li, and D.E. Pelinovsky, ``Counting unstable eigenvalues in Hamiltonian spectral problems via commuting operators", \textit{Comm. Math. Phys.} {\bf 354} (2017) 247--268.
	
	\bibitem{Hristov} N. Hristov and D. E. Pelinovsky, ``Justification of the KP-II approximation in dynamics fo two-dimensional FPU systems", \textit{ZAMP} {\bf 73} (2022) 213 (26 pages).

	\bibitem{Johnson} R.S. Johnson, ``Camassa--Holm, Korteweg--de Vries and related models for water waves", J. Fluid Mech. {\bf 455} (2002) 63--82.
	
	\bibitem{KP-1970} B.B. Kadomtsev and V.I. Petviashvili, ``On the stability of solitary waves in weakly dispersing media", \textit{Sov. Phys. Dokl.} {\bf 15} (1970) 539--541.
	
\bibitem{LP-21}	S. Lafortune and D.E. Pelinovsky, ``Spectral instability of peakons in the b-family of the Camassa-Holm equations", \textit{SIAM J. Math. Anal.} {\bf 54} (2022) 4572--4590
	
	\bibitem{LP-22} S. Lafortune and D.E. Pelinovsky, ``Stability of smooth solitary waves in the $b$-Camassa--Holm equations", \textit{Physica D} {\bf 440} (2022) 133477 (10 pages).
	
	\bibitem{Len1} J. Lenells, ``Stability of periodic peakons", \textit{Int. Math. Res. Not.} {\bf 2004} (2004), 485--499.
	
	\bibitem{Len2} J. Lenells, ``A variational approach to the stability of periodic peakons",
	\textit{J. Nonlinear Math. Phys.} {\bf 11} (2004) 151--163.
	
	
	\bibitem{Lenells} J.~Lenells, ``Traveling wave solutions of the Camassa-Holm equation", \textit{J. Diff. Eq.}  {\bf 217(2)} (2005) 393--430. 
	
	\bibitem{Len4} J. Lenells, ``Stability for the periodic Camassa--Holm
	equation", \textit{Math. Scand.} {\bf 97} (2005) 188--200.
	
	\bibitem{Linares} F. Linares, G. Ponce, and Th. C. Sideris, ``Properties of solutions to the Camassa--Holm
	equation on the line in a class containing the peakons", \textit{Advanced Studies in Pure Mathematics} {\bf 81} (2019), 196--245.
	
	\bibitem{MP20} A. Madiyeva and D.E. Pelinovsky, ``Growth of perturbations to the peaked periodic waves in the Camassa-Holm equation", \textit{SIAM J. Math. Anal.}  {\bf 53} (2021), 3016--3039.
	
	\bibitem{M} T. Mizumachi, ``Stability of line solitons for the KP-II equation in $\mathbb{R}^2$", \textit{Mem. Amer. Math. Soc.} {\bf 238} (2015), no. 1125 (95 pages).
	
\bibitem{MT} T. Mizumachi and N. Tzvetkov, ``Stability of the line soliton of the KP-II equation under periodic transverse perturbations", \textit{Math. Ann.}
{\bf 352} (2012)  659--690.
		
		\bibitem{MS1} T. Mizumachi and Y. Shimabukuro, ``Asymptotic linear stability of Benney--Luke line solitary waves in 2D",
		\textit{Nonlinearity} {\bf 30} (2017) 3419--3465
	
	\bibitem{MS2} T. Mizumachi and Y. Shimabukuro, ``Stability of Benney--Luke line solitary waves in 2 dimensions",
	\textit{SIAM J. Math. Anal.} {\bf 52} (2020)  4238--4283
	
	\bibitem{NP} F. Natali and D.E. Pelinovsky, ``Instability of $H^1$-stable peakons in the Camassa--Holm equation", \textit{J. Diff. Eqs.} 
	{\bf 268} (2020), 7342--7363.
	
	\bibitem{PelSchn} D.E. Pelinovsky and G. Schneider, ``KP-II approximation for a scalar FPU system on a 2D square lattice", \textit{SIAM J. Appl. Math.} {\bf 83} (2023) 	79--98.
	
		\bibitem{PW92} R. Pego and M. I. Weinstein, ``Eigenvalues, and instabilities of solitary waves", \textit{Philos. Trans. Roy. Soc. London Ser. A} {\bf 340} (1992) 47--94.
		
	\bibitem{PW94} R. Pego and M. I. Weinstein, ``Asymptotic stability of solitary waves", \textit{Comm. Math. Phys.} {\bf 164} (1994) 305--349
	
\bibitem{W} A. Welters, ``On explicit recursive formulas in the spectral perturbation analysis of a Jordan block", \textit{SIAM J. Matrix Anal. Appl.} {\bf 32} (2011) 1--22

	\bibitem{Liu2023} Q. Zhang, Y. Xu, and Y. Liu, ``A discontinuous Galerkin method for the Camassa--Holm--Kadomtsev--Petviashvili type
equations", \textit{Numer. Methods PDE} (2023)
\end{thebibliography}

\end{document}